\pdfoutput=1
\documentclass{article}

\usepackage{amsmath,amssymb,amsthm}
\usepackage{thomas}
\usepackage{tznotes}
\usepackage{reference}
\usepackage{url}
\usepackage[latin1]{inputenc}
\usepackage{mdwlist}
\usepackage{refcount}
\usepackage{ifthen}
\usepackage{etoolbox}

\makeatletter
\newcommand{\subjclass}[2][2010]{%
  \let\@oldtitle\@title%
  \gdef\@title{\@oldtitle\footnotetext{#1 \emph{Mathematics subject classification.} #2}}%
}
\newcommand{\keywords}[1]{%
  \let\@@oldtitle\@title%
  \gdef\@title{\@@oldtitle\footnotetext{\emph{Key words and phrases.} #1.}}%
}
\makeatother

\newcommand{\Hinf}{(H$\infty$)}
\newcommand{\CH}{CH}

\newif\ifpromise

\newtoggle{Steve}
\togglefalse{Steve}
\newtoggle{SF}
\togglefalse{SF}

\begin{document}
\author{Thomas Zürcher}
\title{Space-Filling vs.\
Luzin's Condition~(N)}
\maketitle
\begin{abstract}
Let us assume that we are given two metric spaces $(X,d_X)$ and $(Y,d_Y)$ where the Hausdorff dimension $s$ of $X$ is strictly smaller than that of $Y$. Suppose that $X$ is $\sigma$\nobreakdash-finite with respect to $\mathcal{H}^s$. Then we show that for quite general metric spaces, if $f\colon X\to Y$ is a measurable surjection, there is a set $N\subset X$ with $\mathcal{H}^s(N)=0$ and $\mathcal{H}^s(f(N))>0$. If $f$ is continuous, then we investigate whether $N$ can be chosen to be perfect.

We also study more general situations where the measures on $X$ and $Y$ are not necessarily the same and not necessarily Hausdorff measures.
\end{abstract}

\tableofcontents
\section{Introduction}
Around 1877,
G.~Cantor realized that there exists a bijection between the unit
interval $[0,1]$ and the unit square $[0,1]^2$. Motivated by this
result, there has been an interest in understanding how mappings can increase the
dimension. E.~Netto proved that it is impossible to find a
continuous bijection between the interval and the square. However, G.~Peano's work showed that one can
construct a continuous surjection from the interval onto the
square. We refer the reader to the book of H.~Sagan for a more
detailed account of the history of space-filling curves,
\cite{Sagan}. Denoting\footnote{We use the symbol $\surj$ to indicate that the mapping under consideration is a surjection.} by $p\colon [0,1]\surj [0,1]^2$ the space-filling curve from Peano's construction, we can define $f\colon [0,1]^2\surj [0,1]^2$ by $f(x,y)=p(x)$ and see
that there are  one-dimensional
sets, the intervals $[0,1]\times\{x\}$, \mbox{$x\in [0,1]$}, for example, that are mapped onto
two-dimensional sets. This violates \emph{Luzin's condition~(N)}, which requires that sets of measure zero be mapped to sets of
measure zero. Luzin's condition~(N) is important in applications such as elasticity, see
e.g.\ \cite{Cavitation}. It is also a requirement for various area- and change of variables-formulas to hold, see for example Proposition~1.1 in \cite{MalyTheArea}.

We take the following question as
starting point for our inquiry: \lq\lq Given a continuous
surjection $f\colon [0,1]\surj [0,1]^2$ is there a set $N\subset
[0,1]$ such that $\mathcal{H}^1(N)=0$ and
$\mathcal{H}^1(f(N))>0$?\rq\rq\ The fact that the target is higher
dimensional than the domain means philosophically that some subsets of the domain are
\lq\lq blown up\rq\rq. Luzin's condition~(N) requires that \emph{small
sets stay small}. However, space-filling maps could \emph{still}
satisfy Luzin's condition~(N) by mapping \emph{only} sets of
positive $\mathcal{H}^1$\nobreakdash-measure onto two-dimensional
sets.

In \cite{Paper2,Paper1,Paper3}, we (K.~Wildrick and the author) have studied space-fillings and Luzin's condition~(N) with
respect to so called\footnote{These spaces have also been studied in \cite{Ranjbar} and \cite{romanov_absolute_2008} in the metric setting and in \cite{kauhanen_functions_1999} in the Euclidean setting.}  \emph{Sobolev-Lorentz spaces}. These are
generalizations of Sobolev spaces, spaces of mappings that possess some sort of derivatives that have a certain integrability.
We further argued that a space-filling in a Sobolev-Lorentz space
cannot satisfy Luzin's condition~(N), the reason being that we can
partition the space in a null set and countably many sets where the restrictions of the
mapping are Lipschitz. As Lipschitz mappings do not increase the
dimension, it is necessary that the null set is mapped to a set of
higher dimension. The following criterion is basically Theorem~1
in \cite{Troyanov}:
\begin{lemma}[Theorem~1 in \cite{Troyanov}]\tlabel{Troyanov}
Let $f\colon X\to Y$ be a mapping between two metric spaces. Suppose that $X$
is locally compact and separable and that $\mu$ is a Radon measure
on $X$. Then the following conditions are equivalent:
\begingroup
\renewcommand{\theenumi}{\alph{enumi}}
\def\p@enumii{\theenumi}
\def\labelenumi{(\theenumi)}
\begin{enumerate}
\item There exists a measurable function $w\colon X\to [0,\infty]$ that is
finite almost everywhere and such that
\begin{equation*}
d_Y(f(x_1),f(x_2))\leq d_X(x_1,x_2)(w(x_1)+w(x_2))
\end{equation*}
for all $x_1$ and $x_2$ in $X$.
\item \label{TroyanovPartition}There exists a monotone sequence of compact subsets $K_1\subset K_2\subset \cdots \subset X$
such that $f\trestriction K_k$ is Lipschitz and $\mu(X\setminus \cup_{k=1}^\infty K_k)=0$.
\end{enumerate}
\endgroup
\end{lemma}

In this article, we further investigate the dichotomy between
Luzin's condition~(N) and space-fillings on a more abstract level
than in our previous work. Especially, we look also at settings where \eqref{TroyanovPartition} of \tref{Troyanov} is not available.

More precisely, we will look at mappings of which we me merely know that they are continuous or sometimes even only measurable.

To give a flavor of what kind of results we obtain, we state now one of the results of this article and a corollary. The terms will be explained in the course of the article.
\begin{theorem}\tlabel{MainInfinite}\label{MainInfiniteLab}
Let $(X,\mathfrak{T})$ be a regular, second countable topological
space. Let $\mu$ be a Borel measure on $X$ that is
$\mathcal{G}_\delta$\nobreakdash-regular and such that $X$ is
$\sigma$\nobreakdash-finite with respect to $\mu$. We further require that $X$ can be written as countable union of compact sets and a set of measure zero.

Suppose that $(Y,d)$ is a metric space with Hausdorff measure
$\mathcal{H}^h$ that is of finite order, $h(0)=0$, and assume that $Y$ is not
$\sigma$\nobreakdash-finite with respect to $\mathcal{H}^h$.

If $f\colon X\surj Y$ is a measurable surjection, then there
exists a set $N\subset X$ with $\mu(N)=0$ and
$\mathcal{H}^h(f(N))>0$.
\end{theorem}

As corollary, we obtain the following:
\begin{corollary}\tlabel{Dimension}
Assume that $s>0$ and $(X,d_X)$ is separable and $\sigma$\nobreakdash-finite with respect to
$\mathcal{H}^s$ and can be written as countable union of compact
sets and a set of measure zero. Suppose that the metric space $(Y,d_Y)$ has dimension
$t$. If $f\colon X\surj Y$ is a measurable surjection, then there exists a set $N\subset X$ such that
\mbox{$\mathcal{H}^s(N)=0$} and $\dim_Hf(N)=t$.
\end{corollary}

Based on \cite{HajlaszTyson} and \cite{Paper2}, we will give a construction scheme for space-fillings between rather general metric measure spaces that map a perfect set of measure zero to the whole target space, see \tref{SpaceFilling}. This leads to the question if we can always find a (perfect) set of measure zero that is mapped to the whole space. However, for example in the case $f\colon [0,1]\surj [0,1]^2$, this fails if $f$ is $1/2$\nobreakdash-H\"older, which may happen:
\toggletrue{Steve}
\begin{theorem}[Theorem~3 in \cite{BuckleyPeanoPre}]\tlabel{Buckley}
There exist Peano curves $F\colon [0,1]\to [0,1]^2$ that are
$\alpha$\nobreakdash-H\"older continuous for $\alpha=1/2$, but no
such curve is $\alpha$\nobreakdash-H\"older continuous for
$\alpha>1/2$.
\end{theorem}

For related results, see also \cite{Shchepin}, where H\"older continuous surjections between cubes are studied. Other counter examples can be constructed from the following result from \cite{Thread}:
\begin{theorem}[Thread Theorem in \cite{Thread}]
For each $n\geq 2$ there exists a continuous, one-to-one mapping $\varphi\colon \closedopen{0}{1}\to (0,1)^n$
such that $\leb^1(\varphi^{-1}(B))=\leb^n(B)$ for all Borel subsets $B$ of $[0,1]^n$.
\end{theorem}

For each $n\geq 2$, we obtain a continuous bijection $f\colon \closedopen{0}{1}\to f(\closedopen{0}{1})\subset (0,1)^n$
such that $\leb^1(B)=\leb^n(f(B))$ for every Borel set $B\subset \closedopen{0}{1}$, and thus $\leb^1(N)=0$ if
and only if $\leb^n(f(N))=0$. Note that in the present case $\leb^n(f(\closedopen{0}{1}))>0$.

The following \nref{StandasExample} was explained to the author by
K.~Wildrick, who learned it from S.~Hencl.
\begin{example}\tlabel{StandasExample} If $f\colon [0,1]\surj [0,1]^2$ is a continuous
surjection, then there exists a set \mbox{$N\subset [0,1]$} with
$\mathcal{H}^1(N)=0$ and $\mathcal{H}^1(f(N))>0$. Moreover, $N$ can be chosen to be closed.
\end{example}

\begin{proof}
Suppose that the above statement is not true. Hence, for every set \mbox{$N\subset
[0,1]$} with $\mathcal{H}^1(N)=0$, we obtain
$\mathcal{H}^1(f(N))=0$. Equivalently put, if \mbox{$\mathcal{H}^1(f(E))>0$},
then $\mathcal{H}^1(E)>0$ for all sets $E\subset [0,1]$. To arrive at a contradiction, it suffices to find uncountably many
pairwise disjoint closed sets $E_\alpha$ with
\mbox{$\mathcal{H}^1(f(E_\alpha))>0$}. Our assumption tells us then that
$\mathcal{H}^1(E_\alpha)>0$, and hence
$\mathcal{H}^1([0,1])=\infty$. We can clearly find uncountably
many pairwise disjoint closed sets $F_\alpha\subset [0,1]^2$ with
$\mathcal{H}^1(F_\alpha)>0$. The sets
$E_\alpha:=f^{-1}(F_\alpha)$ do the job.
\end{proof}
\label{ChapHere}
We
\makeatletter%
\@ifundefined{r@ChapHere}{%
will construct}%
{\ifthenelse{\ref{ChapHere}<\ref{ChapSF}}{will construct}{\ifthenelse{\ref{ChapSF}<\ref{ChapHere}}{constructed}{\gruError}}%
}
\makeatother
in \tref{SpaceFilling} continuous surjections. In these constructions, it is easily observed that $N$ can be chosen to be a perfect set. The following result puts this observation into perspective. Note that it implies that each closed set in a $T_1$\nobreakdash-space can be written as union of a perfect and of a countable set.

\begin{theorem}[Cantor-Bendixson, Theorem~\ignoreSpellCheck{XIV}.5.3 in \cite{Kuratowski}]\tlabel{CantorBendixson} Every $T_1$\nobreakdash-space\footnote{A topological space is called a \emph{$T_1$\nobreakdash-space} if each single element is closed.} with a countable base is the union of two disjoint sets, one dense in itself and closed (i.e.\ perfect) and the other countable.
\end{theorem}

\textbf{Structure of the paper} The paper can roughly be divided into three parts. We start in Section~\ref{Notation and background} with notation, the measure theoretic background, and some properties of metric spaces. The main content of Section~\ref{Abstract Example} is to give an abstract version of \tref{StandasExample}. An important property that we need in this abstract result is the existence of a certain amount of pairwise disjoint sets. In Section~\ref{First criterion}, we give a first result guaranteeing the needed amount of pairwise disjoint subsets. The following section concludes the first part of the article by providing a first result about blowing up sets.

Unlike in the first part, where we focussed on conditions on $Y$ to obtain sets that are blown up, in the second part, we impose more conditions on $X$. We do this by searching conditions such that $X$ can be written as union of compact sets with finite measure and a set of measure zero, such that the restrictions of the space-filling under question to the compact sets are continuous. In Section~\ref{Topology}, we turn our focus to such a partition of $X$. In Section~\ref{From measurable to continuous}, our thoughts center around Luzin's theorem to improve above partition such that the restriction of the mapping to the sets of positive measure are continuous. This second part is concluded with the proof of \tref{MainInfinite} in Section~\ref{Section proof of main result}.

The following sections constitute the last part. They complement the findings in the first two parts. Having assumed the existence of space-fillings in the first two parts, we give in Section~\ref{Section space fillings} a result showing that there are plenty of space-fillings. Section~\ref{Space fillings and condition N} complements the existence of blown up sets by considering space-fillings that do satisfy Luzin's condition~(N). In some respect, these results show the sharpness of some of the assumptions in our main results. We conclude the article with Section~\ref{Applications}, where we describe some applications of our results.

\textbf{Acknowledgements:}  I would like to thank Kevin Wildrick, Pekka Koskela, Pertti Mattila, Tapio Rajala, Ville Tengvall, and Jeremy Tyson for stimulating discussions about the subject of the article. Thanks also go to Stanislav Hencl, to whom \tref{StandasExample} goes back. I am grateful to the people in Jyv\"asklyl\"a and Helsinki who attended my talks about the subject, and to the Department of Mathematics and Statistics in Jyväskylä and the Mathematical Institute in Bern. This work was supported by the Swiss National Science Foundation grant PBBEP3\_130157 and the Academy of Finland grant 251650. I thank for the received support.

\section{Notation, measure theoretic background, and some properties of metric spaces}\label{Notation and background}
In this section, we lay out some measure theoretic facts. As
sources, we mainly use \cite{Rogers} and \cite{Howroyd} as
reference for the Hausdorff measures and D.~H.~Fremlin's opus on measure
theory \cite{Fremlin}.

We denote by \textbf{measure} what some other authors, Fremlin for example, call
\textbf{outer measure}. If we cite a result from one of his volumes, then we
replace a possible occurrence of a $\sigma$\nobreakdash-algebra by
the $\sigma$\nobreakdash-algebra of the measurable sets.
\begin{definition}[measure, $\sigma$\nobreakdash-finite, (Borel) measurable]\tlabel{measureDef}
Let $X$ be a set.
\begin{itemize}
\item[(a)] If $\mu\colon \mathcal{P}(X)\to [0,\infty]$ is such that
\begin{itemize}
\item[$\bullet$] $\mu(\emptyset)=0$,
\item[$\bullet$] $\mu(A)\leq \sum_{i=1}^\infty \mu(A_i)$ if $A\subset
\cup_{i=1}^\infty A_i$,
\end{itemize}
then we say that $\mu$ is a \emph{measure}, and we call $(X,\mu)$ a
\emph{measure space}.
\item[(b)] If $\mu$ is a measure on $X$, a set $M$ is said to be \emph{$\mu$\nobreakdash-measurable}\footnote{If the measure under consideration is clear, then we also speak simply of \emph{measurable}. We will also use the equivalent formulation that $M$ is measurable if and only if for every set $F$, we have $\mu(F)=\mu(F\cap M)+\mu(F\setminus M)$.} if for all sets $A$, $B$ with $A\subset M$ and $B\subset X\setminus M$, we have $\mu(A\cup B)=\mu(A)+\mu(B)$. We say that $\mu$ is \emph{Borel measurable} if each Borel set is measurable.
\item[(c)] We say that a set $S$ in a measure space $(X,\mu)$ is \emph{$\sigma$\nobreakdash-finite} with respect to $\mu$ if $S$ can be written as countable union of measurable subsets with finite measure.
\item[(d)] Let $(X,\mu)$ be a measure space and
$(Y,\mathfrak{S})$ a topological space. We say that a mapping
$f\colon X\to Y$ is \emph{measurable} if for any open set $O\in \mathfrak{S}$, its preimage
$f^{-1}(O)$ is measurable; if $f^{-1}(O)$ is a Borel set, we say that $f$ is \emph{Borel measurable}.
\end{itemize}
\end{definition}

The following fact can be found for example in Section~112C in
\cite{Fremlin}.
\begin{lemma}\tlabel{measureSup} Let $(X,\mu)$ be a measure space. If
$(M_n)_n$ is a non-decreasing sequence of measurable sets (that
is, $M_n\subset M_{n+1}$, $n\in \field{N}$), then
\begin{equation*}
\mu(\cup_n M_n)=\lim_{n\to \infty} \mu(M_n)=\sup_n \mu(M_n).
\end{equation*}
\end{lemma}

We take the following definition from Definition~9 in \cite{Rogers} and
411B and 411D in \cite{Fremlin}:
\begin{definition}[$\mathcal{R}$\nobreakdash-regular, inner regular, and outer regular]
If $\mathcal{R}$ is a class of sets, a measure $\mu$ is said to be
\begin{itemize}
\item[(a)] \emph{$\mathcal{R}$\nobreakdash-regular}, if for each $E$ in
$X$ there is a set $R$ in $\mathcal{R}$ with $E\subset R$ and
$\mu(E)=\mu(R)$,
\item[(b)] \emph{inner regular with respect to $\mathcal{R}$} if
\begin{equation*}
\mu(M)=\sup\{\mu(R):\, R\in \mathcal{R},\, R\subset M,\, \text{and
$R$ measurable}\}
\end{equation*}
for every measurable set $M$,
\item[(c)] \emph{outer regular with respect to $\mathcal{R}$} if
\begin{equation*}
\mu(M)=\inf\{\mu(R): R\in \mathcal{R},\, R\supset M,\, \text{and
$R$ measurable}\}
\end{equation*}
for every measurable set $M$.
\end{itemize}
\end{definition}

\begin{definition}[premeasure, finite order]
A \emph{premeasure} $\xi$ on $Y$ is a function  mapping the
subsets of $Y$ to the non-negative reals satisfying
\begin{itemize}
\item[(a)] $\xi(\emptyset)=0$,
\item[(b)] if $U\subset V$ then $\xi(U)\leq \xi(V)$ for all $U,\,
V\subset Y$.
\end{itemize}
We will say that the premeasure $\xi$ is of \emph{finite order} if
and only if for some constant $\eta$, we have
\begin{itemize}
\item[(c)] $\xi(\widehat{U})\leq \eta \xi(U)$ for all $U\subset Y$,
\item[(d)] $\inf\{\xi(B(y,\delta)):\, \delta>0\}\leq \eta\xi(\{y\})$
for all $y\in Y$,
\end{itemize}
where
\begin{equation*}
\widehat{U}=\bigcup\{E\subset Y:\, E\cap U\not=\emptyset\;
\text{and $\diam E\leq \diam U$}\}
\end{equation*}
and $B(y,\delta)$ denotes the open ball with center $y$ and radius
$\delta$.
\end{definition}

\begin{definition}[$\delta$\nobreakdash-cover]
We say that a sequence $(U_i)_i$ of subsets of $Y$ is a
\emph{$\delta$\nobreakdash-cover} of a set $E$ if and only if
$E\subset \cup U_i$ and
\begin{equation*}
\diam U_i\leq \delta,\ i\in \field{N}.
\end{equation*}
We use $\Omega_\delta(E)$ to denote the family of all such
(countable) $\delta$\nobreakdash-covers of $E$.
\end{definition}

\begin{definition}[Hausdorff measure]
The measures $\Lambda^\xi_\delta$ are defined for $\delta>0$ by
\begin{equation*}
\Lambda^\xi_\delta(E)=\inf\{\sum_{i=1}^\infty \xi(U_i):\,
(U_i)_i\in \Omega_\delta(E)\},
\end{equation*}
with the convention that $\inf \emptyset=\infty$. The
\emph{Hausdorff $\xi$\nobreakdash-measure} $\Lambda^\xi$ is then
defined as $\Lambda^\xi(E):=\sup_{\delta>0} \Lambda^\xi_\delta(E)$.
\end{definition}

\begin{definition}[Hausdorff function, finite order]\tlabel{Def Hausdorff}
A function $h$, defined for all non-negative real numbers, is a
\emph{Hausdorff function} if and only if the following conditions
are satisfied:
\begin{enumerate}
\item[(a)] $h(t)>0$ for all $t>0$,
\item[(b)] $h(t)\geq h(s)$ for all $t\geq s$,
\item[(c)] $h$ is continuous from the right for all $t\geq 0$.
\end{enumerate}
For such a function and a positive constant $\Theta$, we define a
premeasure, $\xi$ say, on $X$ by
\begin{equation}\label{definitionOfXi}
\begin{split}
\xi(U)&=\min\{h(\diam U),h(\Theta)\},\quad U\not=\emptyset,\\
\xi(\emptyset)&=0.
\end{split}
\end{equation}
If the premeasure $\xi$ is defined by some Hausdorff function
$h$, we will use $\mathcal{H}^h$ for $\Lambda^\xi$. If
\begin{equation*}
\limsup_{t\to 0+} \frac{h(3t)}{h(t)}<\infty,
\end{equation*}
we say that $h$ is \emph{of finite order}.
\end{definition}

\begin{remark}$ $
\begin{itemize}
\item The constant $\Theta$ ensures that the premeasure assigns finite
values to all sets. We can allow for Hausdorff functions that admit the value $\infty$ if it is possible to choose $0<\Theta$ so small that $h(t)<\infty$ for $t\leq \Theta$.
\item If $h$ is of finite order, then the induced premeasure is of
finite order for small enough $\Theta$. When we speak of finite
order, then we assume that $\Theta$ has been chosen to be so small
that $\xi$ is of finite order.
\item  Assume that $h$ is of finite order, continuous from the right in $0$, increasing, $h(t)>0$ for $t>0$, and $h(0)=0$. If we set $H(0)=0$ and for $t>0$
\begin{equation*}
H(t):=\frac{1}{t}\int_t^{2t} h(s)\, ds,
\end{equation*}
then $H$ is a continuous Hausdorff function of finite order comparable to $h$ for small arguments, see for example Section~1 in \cite{Edgar} for more information.
\end{itemize}
\end{remark}

We take the definition of the Hausdorff-Besicovitch dimension from Section~4 in \cite{Howroyd}.
\begin{definition}[Hausdorff-Besicovitch dimension]
We define the \emph{Hausdorff-Besicovitch dimension}\footnote{It is also known as \emph{Hausdorff dimension}. Sometimes we also simply omit both, Hausdorff and Besicovitch, and just talk about dimension.} of a metric space, $(X,d)$, to be the supremum of all non-negative $s$ for which $\Lambda^{h(s)}>0$, where $h(s)$ is defined on all non-negative $t$ by $h(s)(t)=t^s$. We denote the Hausdorff-Besicovitch dimension of $X$ by $\dim_H(X)$.
\end{definition}

We can define a partial strict ordering on the set of Hausdorff functions:
\begin{definition}
We say that
\begin{equation*}
g \prec h
\end{equation*}
for two Hausdorff functions $g$ and $h$ if
\begin{equation*}
\lim_{t\to 0+} \frac{h(t)}{g(t)}=0.
\end{equation*}
\end{definition}

The following \nref{Comparability} is stated on p.~79 in \cite{Rogers} as corollary and enables us to talk about sort of \emph{generalized dimension}.
\begin{lemma}\tlabel{Comparability}
Let $f,\, g$, and $h$ be Hausdorff functions with $f\prec g \prec h$. If a set $E$ in a metric space has $\sigma$\nobreakdash-finite positive $\mathcal{H}^g$\nobreakdash-measure, then $E$ has zero $\mathcal{H}^h$\nobreakdash-measure and non\nobreakdash-$\sigma$\nobreakdash-finite $\mathcal{H}^f$\nobreakdash-measure.
\end{lemma}

Theorem~27 on p.~50 in \cite{Rogers} tells us about
regularity properties of Hausdorff measures:
\begin{theorem}\tlabel{RogersRegular}
A Hausdorff measure $\mathcal{H}^h$ is a regular,
$\mathcal{G}_\delta$\nobreakdash-regular\footnote{A set is termed $\mathcal{G}_\delta$ if it can be written as countable intersection of open sets. If it has a representation as countable union of closed sets, then it is called $\mathcal{F}_\sigma$\nobreakdash-set.} metric\footnote{A metric measure $\mu$ is such that $\mu(A\cup B)=\mu(A)+\mu(B)$, whenever $A$ and $B$ have a positive distance.} measure, all Borel
sets are $\mathcal{H}^h$\nobreakdash-measurable, and each
$\mathcal{H}^h$\nobreakdash-measurable set of finite
$\mathcal{H}^h$\nobreakdash-measure contains an
$\mathcal{F}_\sigma$\nobreakdash-set with the same measure.
\end{theorem}

Let us turn our attention towards metric spaces. The aim is
to introduce the needed definitions in order to state Howroyd's
result on the existence of disjoint subsets of positive measure.
\begin{definition}[analytic, Souslin]
A Hausdorff space is \emph{analytic}, also called \emph{Souslin}, if it is
either empty or a continuous image of $\field{N}_0^{\field{N}_0}$.
\end{definition}
\begin{definition}[Souslin's operation]
Let $S$ be the set $\cup_k \field{N}_0^k$. If $\mathcal{E}$ is a
family of sets, we write $\mathbf{S}(\mathcal{E})$ for the
family of sets expressible in the form\footnote{As Fremlin in \cite{Fremlin}, we can regard a member of $\field{N}_0$ as the set of its predecessors, so that $\field{N}_0^k$ can be identified with the set of functions from $k$ to $\field{N}_0$, and if $\phi\in \field{N}_0^{\field{N}_0}$ and $k\in \field{N}_0$, we can speak of the restriction $\phi\trestriction k\in\field{N}_0^k$.}
\begin{equation*}
\bigcup_{\phi\in \field{N}_0^{\field{N}_0}} \bigcap_{k\geq 1} E_{\phi\trestriction_k}
\end{equation*}
for some family $(E_\sigma)_{\sigma \in S}$ in $\mathcal{E}$.

A family $(E_\sigma)_{\sigma \in S}$ is called a \emph{Souslin
scheme}; the operation
\begin{equation*}
(E_\sigma)_{\sigma\in S}\mapsto \bigcup_{\phi\in
\field{N}_0^{\field{N}_0}}\bigcap_{k\geq 1} E_{\phi\trestriction_k}
\end{equation*}
is \emph{Souslin's operation}. Thus $\mathbf{S}(\mathcal{E})$ is
the family of sets obtainable from sets in $\mathcal{E}$ by
Souslin's operation. If $\mathcal{E}=\mathbf{S}(\mathcal{E})$,
we say that $\mathcal{E}$ is \emph{closed under Souslin's
operation}.
\end{definition}

\begin{definition}[Polish space]
A topological space $X$ is \emph{Polish} if it is separable and its
topology can be defined from a metric under which $X$ is complete.
\end{definition}

\begin{remark} By 423B in \cite{Fremlin}, Polish spaces are analytic.
\end{remark}

\begin{definition}[Souslin-F]
Let $X$ be a topological space. A subset of $X$ is a
\emph{Souslin-F} set in $X$ if it is obtainable from closed
subsets of $X$ by Souslin's operation; that is, it is the
projection of a closed subset of $\field{N}_0^{\field{N}_0}\times X$.

For a subset of $\field{R}^n$, or, more generally, of any Polish
space, it is common to say \lq Souslin set\rq\ for \lq Souslin-F
set\rq.
\end{definition}

The next \nref{Souslin} is a version of
\cite[Theorem~423E]{Fremlin}.
\begin{theorem}\tlabel{Souslin}
Let $(X,\mathfrak{T})$ be an analytic Hausdorff space. For a
subset $A$ of $X$, the following are equivalent:
\begin{itemize}
\item[(i)] $A$ is analytic,
\item[(ii)] $A$ is Souslin-F,
\item[(iii)] $A$ can be obtained by Souslin's operation from the
family of Borel subsets of $X$.
\end{itemize}
\end{theorem}

\begin{definition}[finite structural dimension]
We say that $(Y,d)$ has \emph{finite structural dimension} if and
only if for all positive $\xi$, there exist $N\in \field{N}$
such that every subset of $Y$ of sufficiently small diameter
$\delta$ can be covered by $N$ sets of diameter not greater than
$\xi\delta$.
\end{definition}

\begin{remark} Recall that a space is \emph{doubling} if there is a number $N$ such that each ball can be covered by $N$ balls of half the radius of the original ball. Doubling spaces have finite structural dimension. On the other hand, if a space has finite structural dimension and is compact, then it is doubling.
\end{remark}

\begin{definition}[ultrametric]
The metric space $(Y,d)$ is said to be \emph{ultrametric} if and only if
\begin{equation*}
d(x,z)\leq \max\{d(x,y),d(y,z)\}
\end{equation*}
for all $x$, $y$, and $z$ in $Y$.
\end{definition}

Howroyd lists different kind of spaces for which his result holds. We collect these spaces under a common condition:
\begin{definition}
Assume that a space $A=(A,d,\mathcal{H}^h)$ is given.
\begin{itemize}
\item[(H)] We say that $A$ has property (H) if and only if $(A,d)$ is an analytic subspace of a complete, separable metric space $(Y,d)$, and $h$ is a continuous Hausdorff function with $h(0)=0$ such that at least one of the following properties holds:
    \begin{itemize}
    \item[(a)] $h$ is of finite order,
    \item[(b)] $Y$ has finite structural dimension,
    \item[(c)] $Y$ is ultrametric.
    \end{itemize}
\item[\Hinf] We say that $A$ has property \Hinf\ if and only if $A$ satisfies property (H) and is not $\sigma$\nobreakdash-finite with respect to $\mathcal{H}^h$.
\end{itemize}
Assume $X=(X,\mu)$ is a measure space and $\kappa$ a cardinal.
\begin{itemize}
\item[(A$\kappa$)] We say that $X$ satisfies property (A$\kappa$) if and only if it can be written as union of $\kappa$\nobreakdash-many measurable sets of finite measure and is further such that the union of $\kappa$\nobreakdash-many sets of measure zero has measure zero as well.
\end{itemize}
\end{definition}

\begin{remark}[Continuum Hypothesis (\CH)]
Sometimes we will assume the Continuum Hypothesis (\CH), i.e.\ that every infinite subset of $\field{R}$ has either the same cardinality as $\field{R}$ or $\field{Q}$.
\end{remark}

\section{An abstract result}\label{Abstract Example}
It is surprising how few properties of $[0,1]$, $[0,1]^2$ and
their standard metric and measure, we actually need for the
argument in \tref{StandasExample} to work. We start by reviewing the \nref{StandasExample} in a
very abstract fashion\promisetrue. More precisely, we look at $[0,1]$ and
$[0,1]^2$ as topological spaces. We further replace the Hausdorff
measures by more general measures, for example Borel measures. In \tref{StandasExample}, we have
that the target $Y$ is higher dimensional than the domain $X$. This is encoded by the fact that $[0,1]$ has finite $\mathcal{H}^1$\nobreakdash-measure and $[0,1]^2$ is not $\sigma$\nobreakdash-finite with respect to $\mathcal{H}^1$. In the new setting, we do not compare the \lq\lq dimensions\rq\rq\ of
the domain and the target. However, we express the fact that the
domain is \lq\lq smaller\rq\rq\ than the target by
requiring that the domain can be written as union of, say, $\kappa$\nobreakdash-many measurable sets of finite measure, whereas $Y$ cannot. This also permits to look at situations where the image is lower dimensional than the target. If we look for example at a continuous surjection from $[0,1]^3$ onto $[0,1]^2$, then there is a set of measure zero in $[0,1]^3$ that is mapped onto a set of dimension two.

Let us finally state a quite abstract result:

\begin{lemma}\tlabel{StandaAbstract}
Assume that $\kappa<\lambda$ are two cardinals and if $\eta$ is a cardinal with $\kappa\cdot \eta=\lambda$, then $\aleph_0<\eta$. Suppose that $(X,\mu)$ is a measure space satisfying (A$\kappa$).
Assume $(Y,\nu,\mathfrak{S})$ is a measure space equipped with topology $\mathfrak{S}$ that contains $\lambda$\nobreakdash-many pairwise disjoint Borel sets $Y_j$ that have positive measure.

Further, we stipulate the existence of a measurable surjection $f\colon X\surj Y$.

Then there is a set $N\subset X$ such that $\mu(N)=0$ and $\nu(f(N))>0$. Moreover, $N$ can be chosen to be the preimage of one of the sets $Y_j$.
\end{lemma}

\begin{proof}
Without loss of generality, we may assume that $\mu(X)>0$ for otherwise we can choose $N=f^{-1}(Y_1)$.
We first write $X$ as union
\begin{equation*}
X=\cup_{i\in I} X_i,
\end{equation*}
where each $X_i$ is measurable, has finite measure, and  $\card{I}\leq \kappa$. We let $\{Y_j\}_j$ be a collection of $\lambda$\nobreakdash-many pairwise disjoint Borel sets $Y_j$ with $\nu(Y_j)>0$. We set $A_j:=f^{-1}(Y_j)$ and note that the sets $A_j$ are measurable. We further argue that the sets in $\{A_j\}$ are pairwise disjoint. If not, then there are distinct $j$ and $k$ and a point $x\in X$ contained in $A_j\cap A_k$. But then, as $f$ is surjective, $f(x)\in f(A_j)\cap f(A_k)=Y_j\cap Y_k$---a contradiction.

Assume by contradiction that $\mu(N)=0$ implies $\nu(f(N))=0$ for all sets $N\subset X$. Otherwise stated, if $\nu(f(E))>0$, then $\mu(E)>0$ for each measurable set $E\subset X$. This establishes that each set $A_j$ has positive measure.

For each set $A_j$, there is a set $X_i$ such that $\mu(X_i\cap A_j)>0$ for otherwise the measure of $A_j$ would be zero. Having the different cardinalities in mind, we find a set $X_n$ such that there are uncountably many $A_j$ such that $\mu(X_n\cap A_j)>0$. There exists $m\in \field{N}$ and a countably infinite set $L$ such that for the above chosen set $X_n$
\begin{equation*}
\mu(A_l\cap X_n)>\frac{1}{m}
\end{equation*}
for $l\in L$.

This implies that
\begin{equation*}
\infty>\mu(X_n)\geq\sum_{l\in L}\mu(X_n\cap A_l)\geq \sum_{l\in L}\frac{1}{m}=\infty.
\end{equation*}
This contradiction gives the proof.
\end{proof}

\begin{remark}
A.~J.~Ostaszewski shows in Theorem~2 of \cite{Ostaszewski} that Martin's axiom implies that the union of less than $2^{\aleph_0}$ sets of $\mu$\nobreakdash-measure zero is of $\mu$\nobreakdash-measure zero if $\mu$ is a measure of Hausdorff type. He attributes the proof for the case when $\mu$ is a Hausdorff measure $\mathcal{H}^h$ with Hausdorff function $h$ satisfying $h(0)=0$ to Martin and Solovay, \cite{MartinSolovay}. Further studies of the connection between Martin's axiom and Hausdorff measures can be found in \cite{Zindulka}.
\end{remark}

The conclusion of \tref{StandaAbstract} can hold in cases where the assumptions are not satisfied:
\begin{example}
Let $Y=\{y_0\}$ be a set with one element. Define the Borel measure $\nu$ on $Y$ by
$\nu(Y)=\infty$. Then $Y$ is not $\sigma$\nobreakdash-finite with respect to
$\nu$, but does not contain uncountably many pairwise disjoint sets. If $(X,\mu)$
is a measure space that contains a point $\{x_0\}$ that has measure zero
and $f\colon X\surj Y$ is a surjection, then there exists a set $N\subset X$
with $\mu(N)=0$ and $0<\nu(f(N))=\infty$. For example, we can take $N=\{x_0\}$.
\end{example}

However, we will see in \tref{trivialTopology} that we need some conditions in order that the conclusion of \tref{StandaAbstract} holds.\footnote{at least when we assume the Continuum Hypothesis}

\section{A first test for the existence of disjoint sets}\label{First criterion}
In \tref{StandaAbstract}, we require the existence of a certain amount of pairwise disjoint Borel sets $B_\alpha$ in $Y$
with $\nu(B_\alpha)>0$. Later, we will assume that $Y$ is a
metric space and $\nu$ a Hausdorff measure; but before, we provide
an abstract criterion for the existence of the desired sets. We will
employ it later in concrete situations.

The idea is the following: we assume that $\nu(Y)=\infty$ and
that $Y$ cannot be written as, let us say, a countable union $\cup_n Y_n$ of
Borel sets $Y_n$ with $\nu(Y_n)<\infty$. We find a Borel set $B_1\subset Y$ with $0<\nu(B_1)<\infty$. In the next
step, we look at $Y\setminus B_1$ and extract a Borel set
$B_2\subset Y\setminus B_1$ with $0<\nu(B_2)<\infty$.
Continuing this process, we end up with countably many
pairwise disjoint Borel sets $B_n$ with $0<\nu(B_n)$. To
prove the existence of a desired collection with uncountably many
elements, we resort to Zorn's Lemma. We prove that there has to be
a maximal collection of pairwise disjoint Borel sets,
and that this collection has the desired cardinality.

Let us recall Zorn's Lemma, which is equivalent to
the axiom of choice:
\begin{lemma}[Zorn's Lemma]\tlabel{Zorn} Let
$(P,\leq)$ be a nonempty partially ordered set, i.e.\ for all
$a,\, b,\, c\in P$
\begin{enumerate}
\item[(a)] $a\leq a$,
\item[(b)] $a\leq b$ and $b\leq a$ implies $a=b$,
\item[(c)] $a\leq b$ and $b\leq c$ implies $a\leq c$.
\end{enumerate}
A \emph{chain} $Q$ in $P$ is a subset of $P$ such that for all
$a,\, b\in Q$, $a\leq b$ or $b\leq a$. If each chain in $P$ has an
upper bound in $P$, then $P$ has a maximal element.
\end{lemma}

In the following \nref{UncountablyManyBorelPlusSets}, we can for example think of $\mathcal{A}$ as the analytic sets, $I$ as $\field{N}$ and $\mathcal{C}$ as the collection of the Borel or the closed sets.

\begin{lemma}\tlabel{UncountablyManyBorelPlusSets}
Let $(Y,\nu)$ be a measure space and $\mathcal{C}$ be a collection of subsets of $Y$.
Assume there exists
a collection $\mathcal{A}$ of subsets of $Y$ and a cardinal $\kappa\geq \aleph_0$ with the following
properties
\begingroup
\renewcommand{\theenumi}{\alph{enumi}}
\def\p@enumii{\theenumi}
\def\labelenumi{(\theenumi)}
\begin{enumerate}
\item $Y\in \mathcal{A}$,\label{YinA}
\item  $Y$ cannot be written as union of $\kappa$\nobreakdash-many sets in $\mathcal{A}$ with finite measure,\label{notSigmaFinite}
\item $\cup_{j\in J} C_j$ and $Y\setminus \cup_{j\in J} C_j$ lie in $\mathcal{A}$ provided $\card(J)\leq \kappa$ and $C_j\in \mathcal{C}$,\label{ComplementOK}
\item If $\nu(A)=\infty$ for some $A\in \mathcal{A}$,
then there exists a set $C\in \mathcal{C}$ with $C\subset A$ and
$0<\nu(C)<\infty$.\label{ExistenceSubsetInA}
\end{enumerate}
\endgroup
Under these conditions, there exist
exists a collection of
pairwise disjoint sets $C_\gamma\in \mathcal{C}$ with $0<\nu(C_\gamma)$ whose cardinality is strictly larger than $\kappa$.
\end{lemma}

\begin{proof}
We let
\begin{equation*}
\mathcal{D}:=\{\{D_\delta\}_{\delta\in J}\}
\end{equation*}
be the set of all possible collections $\{D_\delta\}_{\delta\in J}$
of pairwise disjoint sets in the collection $\mathcal{C}$ with
$0<\nu(D_\delta)<\infty$. We say that $\{D_\delta\}_{\delta\in
J}\leq \{\widetilde{D}_{\tilde{\delta}}\}_{\tilde{\delta}\in \tilde{J}}$ if
$\{D_\delta\}_{\delta\in J}\subset
\{\widetilde{D}_{\tilde{\delta}}\}_{\tilde{\delta}\in \tilde{J}}$. Since
$\subset$ is an ordering relation, $(\mathcal{D},\leq)$ is a partially ordered set.

Let us verify the assumptions of Zorn's Lemma and thus the
existence of a maximal collection in $\mathcal{D}$.

By \eqref{YinA}, the set $Y$ is in $\mathcal{A}$ and by \eqref{notSigmaFinite} we obtain
$\nu(Y)=\infty$. Assumption~\eqref{ExistenceSubsetInA} gives a set $D\in \mathcal{C}$ with $0<\nu(D)<\infty$ implying $\{D\}\in \mathcal{D}\not=\emptyset$.

Let $\{\{D^j_\delta\}_{\delta\in J_j}\}_{j\in J}$ be a chain in
$\mathcal{D}$. We set
\begin{equation*}
\mathcal{U}:=\{D^j_\delta:\, \delta\in J_j,\, j\in J\}
\end{equation*}
and claim that $\mathcal{U}$ is an upper bound of the chain in
$\mathcal{D}$. From the definition of $\mathcal{U}$, it follows at
once that its elements are sets $D^j_\delta$ in $\mathcal{C}$  with
$0<\nu(D^j_\delta)<\infty$. Now let us show that the sets in
$\mathcal{U}$ are pairwise disjoint. Let us select two
distinct sets in $\mathcal{U}$. The chain property allows us to assume that both sets are in the same element of the chain forcing
their disjointness. We conclude that $\mathcal{U}$ is indeed in
$\mathcal{D}$, and it is easy to see that $\mathcal{U}$ is an upper
bound for the chain.

Zorn's Lemma~\pref{Zorn} guarantees the existence of a maximal collection
$\{M_\gamma\}_{\gamma\in C}$ of pairwise disjoint sets in $\mathcal{C}$ with
$0<\nu(M_\gamma)<\infty$.

Assume by contradiction that above maximal collection has cardinality at most $\kappa$. We set
\begin{equation*}
M:=\cup_{\gamma\in C}M_\gamma.
\end{equation*}
By \eqref{ComplementOK}, $M$ and $Y\setminus M$ lie in $\mathcal{A}$. According to \eqref{notSigmaFinite}, the measure of $Y\setminus M$ is infinite.
Appealing to \eqref{ExistenceSubsetInA}, there exists a set
$\widetilde{M}\subset Y\setminus M$ in $\mathcal{C}$ with
$0<\nu(\widetilde{M})<\infty$, and it is clearly disjoint to every
element in $\{M_\gamma\}_{\gamma\in C}$. But then the collection
$\{M_\gamma\}_{\gamma\in C}\cup \{\widetilde{M}\}\in \mathcal{D}$
contradicts the maximality of $\{M_\gamma\}_{\gamma\in C}$, and we
are done.
\end{proof}

\begin{remark}
For related theorems, see \cite{DaviesDisjoint} (or the comment before \tref{perfect sets in the compact case}) and Theorem~2 in \cite{LarmanInfinite}.
\end{remark}

\section{Existence of disjoint sets in the metric setting and first results}
We give now a version of Corollary~7 in \cite{Howroyd} that
we will use to verify \eqref{ExistenceSubsetInA} in \tref{UncountablyManyBorelPlusSets}.
\begin{theorem}[Howroyd]\tlabel{ExistenceOfSubsets}
Suppose $(A,d,\mathcal{H}^h)$ satisfies property (H).
Then for all real $l$ with\footnote{See \eqref{definitionOfXi} for the relation between $\xi$ and $h$.} $l<\Lambda^\xi(A)$, there exists a
(compact) subset $K$ of $A$ such that
\begin{equation*}
l<\mathcal{H}^h(K)<\infty.
\end{equation*}
\end{theorem}

\begin{corollary}\tlabel{HowroydCor}
Suppose $(Y,d,\mathcal{H}^h)$ satisfies \Hinf.
Then there are uncountably many pairwise disjoint compact subsets of $Y$ with positive measure.
\end{corollary}

\begin{proof}
In \tref{UncountablyManyBorelPlusSets}, we let $\mathcal{C}$ be the compact and $\mathcal{A}$ the analytic subsets of $Y$. Further, we let $\kappa=\aleph_0$. By \tref{Souslin}, Borel sets are analytic. Requirement~(d) in \tref{UncountablyManyBorelPlusSets} follows from
\tref{ExistenceOfSubsets}, and the other assertions in the \nref{UncountablyManyBorelPlusSets} are easily verified.
\end{proof}

If we assume that the space $Y$ is compact, then we can be more precise about the number of pairwise disjoint sets in \tref{UncountablyManyBorelPlusSets}. The Theorem in \cite{DaviesDisjoint} roughly states that if $(Y,d)$ is a compact space of non-$\sigma$\nobreakdash-finite $\mathcal{H}^h$\nobreakdash-measure, and every closed subset of $Y$ has subsets of finite measure, then $Y$ contains a system of $2^{\aleph_0}$ disjoint closed subsets each of non-$\sigma$\nobreakdash-finite measure.

\begin{corollary}\tlabel{perfect sets in the compact case}
If $(Y,d,\mathcal{H}^h)$ is a compact space that satisfies \Hinf, then there are $2^{\aleph_0}$ pairwise disjoint compact subsets each of non-$\sigma$\nobreakdash-finite measure.
\end{corollary}

We recall the definition of a perfect set.
\begin{definition}[perfect]
A set is \emph{perfect} if it is closed, and each open set that meets it at all, meets it in an infinite set.
\end{definition}

\begin{theorem}\tlabel{GeneralResult} Suppose $X=(X,\mu)$ is $\sigma$\nobreakdash-finite and $\mu$ is Borel. Suppose further that $(Y,d,\mathcal{H}^h)$ satisfies \Hinf.

If $f\colon X\surj Y$ is a measurable surjection, then there is a set $N\subset X$ such that $\mu(N)=0$ and $\mathcal{H}^h(f(N))>0$.

Moreover, if $X$ is additionally equipped with a topology with a countable basis and such that points are closed, and we additionally stipulate that $f$ is such that for every closed set $F$, the preimage $f^{-1}(F)$ can be written as countable union of closed sets, then $N$ can be chosen to be perfect.
\end{theorem}

\begin{proof}
We want to apply \tref{StandaAbstract}. The assumptions on $(X,\mu)$ are satisfied with $\kappa=\aleph_0$. Using \tref{HowroydCor}, we obtain the existence of some cardinal $\lambda>\kappa$ along with $\lambda$\nobreakdash-many pairwise disjoint compact sets $Y_j$ with positive measure.
Applying \tref{StandaAbstract} gives the existence of $N$, and it can be chosen to be the preimage of a closed set.

For the moreover part, note that $N$ can be written as union of countably many closed sets $F_n$. If $\mathcal{H}^h(f(F_n))$ would be zero for all $n\in \field{N}$, then $\mathcal{H}^h(f(N))$ as well. Hence one of the sets $F_n$ is mapped to a set of positive measure. We extract a suitable perfect set by Cantor-Bendixson's \tref{CantorBendixson}.
\end{proof}

If we know that the target is compact, we can relax the assumption that $X$ is $\sigma$\nobreakdash-finite a little bit due to \tref{perfect sets in the compact case}. We skip the proof of the following result.
\begin{theorem}\tlabel{Compact target}
Let us assume that $\kappa<2^{\aleph_0}$ is a cardinal. Suppose that $(X,\mu)$ satisfies (A$\kappa$), and that $(Y,d,\mathcal{H}^h)$ satisfies \Hinf\ and is compact.

If $f\colon X\to Y$ is a measurable surjection, then there is a set $N\subset X$ such that $\mu(N)=0$ and $\mathcal{H}^h(f(N))>0$.

Moreover if $(X,\mu,\mathfrak{T})$ is a topological space where points are closed and with a countable basis, and $f$ is such that for every closed set $F$, the preimage $f^{-1}(F)$ can be written as countable union of closed sets, then $N$ can be chosen to be perfect.
\end{theorem}

\begin{question}
Does \tref{Compact target} also hold if we drop the assumption that $Y$ is compact?
\end{question}

We have encountered \iftoggle{Steve}{}{\todonote{We are referring to something nonexistent}} in \tref{Buckley} a (continuous) mapping $f\colon [0,1]\surj [0,1]^2$ that maps every set $N$ of $\mathcal{H}^1$\nobreakdash-measure zero to a set of $\mathcal{H}^2$\nobreakdash-measure zero. In contrast to \tref{GeneralResult}, there is no (perfect) set $P$ of $\mathcal{H}^1$\nobreakdash-measure zero with $\mathcal{H}^2(f(P))>0$. However, we may try to replace $\mathcal{H}^2(f(P))>0$ by the weaker condition $\dim_H f(P)=2$.

Actually, the next result will be an important ingredient in the proof; however it is stated in greater generality than necessary for the existence of above set. We skip the proof.
\begin{lemma}\tlabel{DimensionFunction}
Let $t>0$ and define
\begin{equation*}
h(x)=
\begin{cases}
0 & x=0,\\
x^t\log(\frac{1}{x^{t/2}}) & 0<x\leq e^{-2/t},\\
e^{-2} & x>e^{-2/t}.
\end{cases}
\end{equation*}
Then $h$ is a continuous Hausdorff function of finite order and $x^{t'}\prec h(x)\prec x^t$ for all $0<t'<t$.
\end{lemma}

\begin{corollary}[Corollary of \tref{GeneralResult}]\tlabel{DimensionBusiness}
Suppose $(X,d)$ is $\sigma$\nobreakdash-finite with respect to $\mathcal{H}^s$. Assume $(Y,d)$ is $t$\nobreakdash-dimensional and an analytic subspace of a complete, separable space. If $f\colon X\surj Y$ is a measurable surjection, then there exists a set $N\subset X$ such that $\mathcal{H}^s(N)=0$ and $\dim_H f(N)=t$.

If additionally $t$ and $\mathcal{H}^t(Y)$ are positive and $f$ is such that for each closed set $F\subset Y$, its preimage can be written as countable union of closed sets, then $N$ can be chosen to be perfect.
\end{corollary}

\begin{proof}
If $t=0$, then $N=\emptyset$ does the job. Otherwise, using \tref{GeneralResult}, we choose for each $n$ with $t-1/n>0$ a set $N_n$ such that $\mathcal{H}^s(N_n)=0$ and $\mathcal{H}^{t-1/n}(f(N))>0$. The union of the sets $N_n$ is as required.
For the second statement, we may by \tref{GeneralResult} assume that $Y$ is $\sigma$\nobreakdash-finite with respect to $\mathcal{H}^t$. We choose $h$ as in \tref{DimensionFunction}. By \tref{Comparability}, $(Y,d,\mathcal{H}^h)$ satisfies \Hinf. \tref{GeneralResult} provides us with a set $N$ (perfect under the additional assumption on $f$) with $\mathcal{H}^s(N)=0$ and $\mathcal{H}^h(f(N))>0$, thus with $\dim_H f(N)=t$ by \tref{DimensionFunction} and \tref{Comparability}.
\end{proof}

We look now at targets that are not separable. These spaces are so
large that it is no problem to find uncountably many of the desired disjoint sets.

For the following \nref{Tarski}, see for example Proposition~1.14 in \cite{Levy}:
\begin{theorem}[Tarski]\tlabel{Tarski}
The axiom of choice is equivalent to the following statement: for
every infinite set $A$, there is a bijection between $A$ and
$A\times A$.
\end{theorem}

\begin{lemma}\tlabel{uncountablePartition}
Every uncountable set $A$ can be written as union of $\card(A)$\nobreakdash-many
pairwise disjoint uncountable sets.
\end{lemma}

\begin{proof}
Assume that $A$ is uncountable. By Tarski's \tref{Tarski}, we find a
bijection $f\colon A\times A\surj A$. We define
\begin{equation*}
\mathcal{A}:=\{A_\alpha:=f(A\times\{\alpha\}),\, \alpha\in A\}.
\end{equation*}
We claim that $\mathcal{A}$ is a partition as described in the statement of the \nref{uncountablePartition}. Given a
set $A_\alpha$, its cardinality is the same as the one of
$A\times\{\alpha\}$ and thus of $A$. Hence $A_\alpha$ is
uncountable. Given distinct $\alpha$ and $\beta$, then
$A\times\{\alpha\}$ and $A\times\{\beta\}$ are disjoint, and since $f$
is injective, their images under $f$ are disjoint as well.
Finally, we see that the cardinality of $\mathcal{A}$ is the same
as the one of $A$.
\end{proof}

\begin{lemma}\tlabel{notSeparable} Assume that $\kappa$ is a cardinal with $\aleph_0\leq \kappa$. Suppose that $(Y,d)$ is a metric space that does not contain any dense set of cardinality less or equal than $\kappa$. Then there
exists a collection of pairwise disjoint closed sets that have infinite measure for every Hausdorff measure $\mathcal{H}^h$, and the cardinality $\lambda$ of the collection is strictly larger than $\kappa$.
\end{lemma}

\begin{proof}
First, we want to construct a set whose cardinality is larger than $\kappa$, and whose elements are
separated so that the measure of the set is infinite. We do this
with the help of Zorn's \tref{Zorn}.

Fix a natural number $n\in \field{N}$. We let
\begin{equation*}
\mathbf{M}_n:=\{M_\alpha\}_\alpha
\end{equation*}
be the collection of all sets $M_\alpha$ such that
if $x$ and $y$ are distinct points in $M_\alpha$, then
$d(x,y)\geq 1/n$. We define an order $\leq$ on $\mathbf{M}_n$ simply by set
inclusion: $C\leq D$ if and only if $C\subset D$. It is clear that $Y$ is not empty, and we can take $y_0\in Y$. Then the
set $\{y_0\}$ is in $\mathbf{M}_n$ showing that $\mathbf{M}_n$ is not empty. Assume that
$\{A_\alpha\}_{\alpha\in J}$ is a chain in $\mathbf{M}_n$. We claim that
\begin{equation*}
\mathcal{U}:=\cup_{\alpha\in J} A_\alpha
\end{equation*}
is an upper bound for the chain in $\mathbf{M}_n$. Given two distinct
points $x$ and $y$ in $\mathcal{U}$, they belong a priori to two
different elements $A_\alpha$ and $A_\beta$. However, by the chain
property, one set is contained in the other, and hence we can
assume that $x$ and $y$ belong to the same set, and hence
$d(x,y)\geq 1/n$. That $\mathcal{U}$ is larger than any element in
the chain is clear. Applying Zorn's Lemma~\pref{Zorn}, we conclude
the existence of a maximal set $\mathcal{M}_n$, whose elements are
all at least $1/n$ apart. We set
\begin{equation*}
\mathcal{M}:=\cup_n \mathcal{M}_n.
\end{equation*}
We assume by contradiction that $\mathcal{M}$ and hence every
$\mathcal{M}_n$ has cardinality at most $\kappa$. The contradiction will follow as soon as we have shown that $\mathcal{M}$ is dense in $Y$. Let $\varepsilon>0$ and $y\in Y$. We
choose a natural number $n\in \field{N}$ such that
$1/n<\varepsilon$. If $y$ would be such that $d(x,y)\geq 1/n$ for
all $x\in \mathcal{M}_n$, then we could add $y$ to $\mathcal{M}_n$
contradicting the maximality of $\mathcal{M}_n$. Hence, there exists $x\in
\mathcal{M}_n$ with $d(x,y)<1/n<\varepsilon$. Since $y$ and
$\varepsilon>0$ were arbitrary, it follows that $\mathcal{M}$ is a
dense subset of $Y$ with cardinality bounded from above by $\kappa$ leading to a contradiction. Hence the cardinality of $\mathcal{M}$ is strictly larger than $\kappa$.
We can conclude that this is true as well for one of the sets $\mathcal{M}_n$.
In \tref{uncountablePartition}, we have verified that we can write
$\mathcal{M}_n$ as union of $\card(\mathcal{M}_n)$\nobreakdash-many pairwise
disjoint uncountable sets $(N_\alpha)_\alpha$. Since any $N_\alpha$ contains
uncountably many points, which are separated, it does not have a
countable cover with sets of diameter smaller than $1/n$, and hence
any Hausdorff measure of $N_\alpha$ is infinite. Furthermore, since $N_\alpha$ consists only of separated
points, it is closed.
\end{proof}

\begin{theorem}\tlabel{GeneralResultNonSeparable}
Suppose that $\kappa\geq \aleph_0$ is a cardinal. We assume that $X=(X,\mu)$ is a measure space satisfying (A$\kappa$). We stipulate that $(Y,d,\mathcal{H}^h)$ is a metric measure space, where $\mathcal{H}^h$ is a Hausdorff measure.

Suppose that any dense set in $Y$ has cardinality strictly larger than $\kappa$.

If $f\colon X\surj Y$ is a measurable surjection, then there is a set $N\subset X$ such that $\mu(N)=0$ and $\mathcal{H}^h(f(N))>0$.

Moreover, if $X$ is additionally equipped with a topology with a countable basis and such that points are closed, and $f$ is additionally such that for every closed set $F$, the preimage $f^{-1}(F)$ can be written as countable union of closed sets, then $N$ can be chosen to be perfect.
\end{theorem}

\begin{proof} The proof is essentially as the one of \tref{GeneralResult}. In \tref{StandaAbstract}, $\lambda$ is as indicated in \tref{notSeparable}, and again the sets $Y_j$ are closed. The moreover-part is proven exactly as in the proof of \tref{GeneralResult}.
\end{proof}

\section{Some topological considerations}\label{Topology}
Until now, we required $Y$ to be an analytic subset of a complete,
separable metric space. We want to get rid of this
restriction. However, the price we pay is the introduction of some additional
regularity conditions on $X$.

They provide us to find a countable partition of $X$ by
\emph{compact} sets $K_n$ (and a set of measure zero) such that the restrictions of $f$ to the
sets $K_n$ are continuous. Therefore, the images $Y_n=f(K_n)$ are compact as well.

Our goal in this section is to show that a
$\sigma$\nobreakdash-finite, $\sigma$\nobreakdash-compact space
can be written as countable union of compact sets with finite measure
and a set of measure zero. To achieve this, we put some restrictions
on the measure and the topology.

The main point in the following \nref{UnionOfCompacts} is the fact that
we can choose the compact sets to have \emph{finite} measure. We will
later on study the existence of the desired
$\mathcal{F}_\sigma$\nobreakdash-sets.

\begin{proposition}\tlabel{UnionOfCompacts}
Let $(X,\mathfrak{T})$ be a topological Hausdorff space and $\mu$ a measure
on $X$ that is $\sigma$\nobreakdash-finite and has the property
that each measurable set of finite measure contains an
$\mathcal{F}_\sigma$\nobreakdash-set with the same measure.
Suppose that $X$ can be written as countable union of compact sets
and a set of measure zero. Then we can write $X$ as countable
union of compact sets with \emph{finite} measure and a set of
measure zero.
\end{proposition}

\begin{proof}
The nature of the statement allows us to assume that $X$ is
$\sigma$\nobreakdash-finite and $\sigma$\nobreakdash-compact.

By definition, we can decompose $X=\cup_n X_n$, where all $X_n$ are measurable and have finite measure. By assumption, we can write
$X_n=\widehat{Z}_n\cup N_n$, where $\widehat{Z}_n$ is an
$\mathcal{F}_\sigma$\nobreakdash-set and $N_n$ has measure zero.
Hence
\begin{equation*}
X=\cup_n Z_n\cup N,
\end{equation*}
where $N$ has measure zero and the sets $Z_n$ are $\mathcal{F}_\sigma$\nobreakdash-sets with finite measure. Hence, each $Z_n$ has the form
\begin{equation*}
Z_n=\cup_l C_l^n,
\end{equation*}
where the sets $C_l^n$ are closed.

On the other hand, we can write
\begin{equation*}
X=\cup_m K_m,
\end{equation*}
where the sets $K_m$ are compact. Now,
\begin{align*}
Z_n&=Z_n\cap \left(\bigcup_m K_m\right)=\bigcup_m(Z_n\cap K_m),\\
Z_n\cap K_m&=\left(\bigcup_l C_l^n\right)\cap K_m=\bigcup_l(C_l^n\cap K_m).
\end{align*}
In topological Hausdorff spaces, compact sets are closed.
As closed subsets of compact sets, the sets $C_l^n\cap K_m$ are compact. It follows that the sets
$Z_n\cap K_m$ and hence $Z_n$ are $\sigma$\nobreakdash-compact.
\end{proof}

\begin{question} Can we weaken the condition in \tref{UnionOfCompacts} that each measurable set of finite measure contains an $\mathcal{F}_\sigma$\nobreakdash-set with the same measure and still obtain the same conclusion?
\end{question}

We further want to study the condition concerning the
$\mathcal{F}_\sigma$\nobreakdash-sets in the last
\nref{UnionOfCompacts}.

The proof of the following result is as the one of Theorem~22 on
p.~35 in \cite{Rogers}:
\begin{theorem}\tlabel{OpenImpliesFSigma}
Let $\mu$ be a Borel measure that is
$\mathcal{G}_\delta$\nobreakdash-regular on a topological space
$(X,\mathfrak{T})$ such that every open set is a
$\mathcal{F}_\sigma$\nobreakdash-set. If $E\subset X$ is
measurable with $\mu(E)<\infty$, then there exists an
$\mathcal{F}_\sigma$\nobreakdash-set $H\subset E$ with
$\mu(H)=\mu(E)$.
\end{theorem}

We now give conditions on a topological space implying that every open
set is an $\mathcal{F}_\sigma$\nobreakdash-set.

\begin{definition}[regular]
A topological space $(X,\mathfrak{T})$ is called \emph{regular} if
points are closed and if for any $x\in X$ and closed set $C\subset
X$ that does not contain $x$, there exist disjoint open subsets $U$ and $V$ of $X$ such
that $x\in U$ and $C\subset V$.
\end{definition}

\begin{remark}\tlabel{NormalHausdorff}
Note that a regular space is Hausdorff.
\end{remark}

\begin{definition}[second countable]
A topological space is \emph{second countable} if it has a
countable basis.
\end{definition}

\begin{lemma}\tlabel{OpenIsFSigma} Let $(X,\mathfrak{T})$ be a regular, second countable
topological space. Then every open set is an
$\mathcal{F}_\sigma$\nobreakdash-set.
\end{lemma}

\begin{proof}
Let $O$ be an open set and choose $x\in O$. We note that
$C:=X\setminus O$ is closed. Since $X$ is regular, there exist
disjoint open sets $U$ and $V$ with $x\in U$ and $C\subset V$. Now
$D:=X\setminus V\subset X\setminus C$ is closed, and since $x$ is
not a point in $V$, it lies in $D$. If $y\in U$, then $y\not\in V$
and hence $y\in D$. Thus
\begin{equation*}
x\in U\subset D\subset X\setminus C=X\setminus(X\setminus O)=O.
\end{equation*}
Consequently, we can choose for every $x\in O$ an open set $U_x$ with $x\in
U_x\subset \overline{U_x}\subset O$.
Let
\begin{equation*}
\mathcal{O}:=\{O_i\}
\end{equation*}
be a countable basis for the topology. Hence, for every $x\in O$ and
$U_x$, we find an open set $O_i$ with $x\in O_i\subset
\overline{O_i}\subset \overline{U_x}\subset O$. We collect all
occurring indices in the set $I$. Then
\begin{equation*}
O=\cup_{i\in I}O_i\subset \cup_{i\in I} \overline{O_i}\subset O,
\end{equation*}
and the claim follows.
\end{proof}

\section{From measurable to continuous}\label{From measurable to continuous}
In Howroyd's \tref{ExistenceOfSubsets}, we have the requirement
that $Y$ is an analytic subset of a complete, separable metric
space. We would like to get rid of some assumptions on $Y$ and add
in turn more requirements on $X$. In this section, we show that $Y$ inherits
the desired topological properties from $X$. For example if $X$ is compact and separable, and $f\colon X\surj Y$ is continuous, then $Y=f(X)$ is compact (thus complete), and separable as well. Thus, we can use the machinery that we developed before.

The following is from Definition~411M in \cite{Fremlin}; however, we have allowed ourselves to change the transcription of Luzin:
\begin{definition}[almost continuous, Luzin measurable]
Let $(X,\mu,\mathfrak{T})$ be a measure space with topology
$\mathfrak{T}$ and $(Y,\mathfrak{S})$ another topological space.
We say that a mapping $f\colon X\to Y$ is \emph{almost continuous}
or \emph{Luzin measurable} if $\mu$ is inner regular with respect
to the family of subsets $A$ of $X$ such that $f\trestriction_A$ is
continuous.
\end{definition}

The main point in\label{meascont} is the verification that our setting permits the application of the following version
of Theorem~451S in \cite{Fremlin}, referenced there to an article by Fremlin and one by Koumoullis and Prikry:
\begin{theorem}[Luzin's theorem]\tlabel{LusinsTheorem}
Let $(X,\mu)$ be a measure space with topology $\mathfrak{T}$ such
that
\begin{itemize}
\item[(a)] whenever a set is measurable with infinite measure, there
exists a subset with positive and finite measure ($\mu$ is
\emph{semi-finite}),
\item[(b)] every point of $X$ has a neighborhood of finite measure
($\mu$ is \emph{locally finite}),
\item[(c)] if $E$ is not measurable, then there exists a measurable set
$F$ of finite measure such that the intersection $E\cap F$ is not
measurable (if additionally $\mu$ is semi-finite, then this is
called \emph{locally determined}),
\item[(d)] the topology is Hausdorff,
\item[(e)] $\mu$ is inner regular with respect to the compact sets.
\end{itemize}
Assume $Y$ is a metrizable space. Then a function $f\colon X\to Y$
is measurable if and only if it is almost continuous.
\end{theorem}

With the help of Luzin's theorem, we obtain a nice partition of
our space:\label{sectionCheck}

\begin{proposition}\tlabel{ContinuousDecomposition} Let $(X,\mathfrak{T})$ be a compact space
where the topology $\mathfrak{T}$ is Hausdorff. Assume that $\mu$ is a Borel measure on
$X$ that is inner regular with respect to the compact sets and
$\mu(X)<\infty$. Suppose $(Y,\mathfrak{S})$ is metrizable. If
$f\colon X\to Y$ is measurable, then we can write
\begin{equation*}
X=\cup_n K_n\cup N,
\end{equation*}
where $K_n$ is compact, $\mu(N)=0$, and $f\trestriction_{K_n}$ is
continuous.
\end{proposition}

\begin{proof}
Note that by Theorem~211L in \cite{Fremlin}, $(X,\mu)$ is locally
determined. The setting is such that we can apply
\tref{LusinsTheorem} to obtain a sequence $(A_m)_m$ of measurable sets
in $X$ such that the restriction $f\trestriction_{A_m}$ of $f$ to
$A_m$ is continuous and
\begin{equation*}
\mu(X)<\mu(A_m)+\frac{1}{2m}.
\end{equation*}
Using the inner regularity of $\mu$ with respect to compact sets,
we obtain compact sets $K_m$ contained in $A_m$ such that
\begin{equation*}
\mu(X)<\mu(A_m)+\frac{1}{2m}<\mu(K_m)+\frac{1}{m}.
\end{equation*}
We set $N:=X\setminus \cup_m K_m$ and deduce that
\begin{equation*}
\mu(N)=\mu(X\setminus \cup_m K_m)=\mu(\cap_m (X\setminus
K_m))\leq \mu(X\setminus K_{m_0})\leq\mu(X)-\mu(K_{m_0})<\frac{1}{m_0}
\end{equation*}
for every $m_0\in \field{N}$. The claim follows.
\end{proof}

\section{\texorpdfstring{Proofs of \stref{MainInfinite} and \stref{Dimension}}{Proof of Theorem~\ref{MainInfiniteLab} and of Corollary~\ref{Dimension}}}\label{Section proof of main result}

\begin{proof}[Proof of \tref{MainInfinite}] First, we want to write $X$ as countable union of
compact sets with finite measure and a set of measure zero.
According to \tref{UnionOfCompacts}, it suffices to verify that
each measurable set of finite measure contains an
$\mathcal{F}_\sigma$\nobreakdash-set of the same measure (regular spaces are Hausdorff).
\tref{OpenImpliesFSigma} reduces the this task to the verification
that open sets are $F_\sigma$\nobreakdash-sets. Under our
assumptions, this follows from \tref{OpenIsFSigma}. Summarizing,
we may now apply \tref{UnionOfCompacts} (we noted in \tref{NormalHausdorff} that $X$ is Hausdorff; the same is true for
the sets $K_n$ in the following decomposition), i.e.\ we may write
\begin{equation*}
X=\cup_n K_n\cup N,
\end{equation*}
where the sets $K_n$ are compact with finite measure, and $N$ has
measure zero. We want to replace each compact set $K_n$ by a union of compact sets and a set of measure zero such that $f$ restricted to each of these new compact sets is continuous. In view of \tref{ContinuousDecomposition}, let us verify the inner regularity with respect to compact sets of the restriction of $\mu$ to $K_n$ denoted by $\mu\trestriction_{K_n}$. If $E\subset K_n$ is
measurable, then we have argued before that it contains an
$\mathcal{F}_\sigma$\nobreakdash-set $F\subset E$ with the same
measure. Hence $F$ can be written as countable union of closed, and since $X$
is compact, of compact sets. Let us denote these compact sets by
$L_m$ and note that, by replacing $L_m$ by the union $\cup_{r=1}^m
L_r$ if necessary, the inclusion $L_m\subset L_{m+1}$ holds. By
\tref{measureSup}, we have
\begin{equation*}
\mu(E)=\mu(F)=\sup_m \mu(L_m),
\end{equation*}
verifying the inner regularity of $\mu\trestriction_{K_n}$ with
respect to compact sets. By \tref{ContinuousDecomposition}, we can
now write
\begin{equation*}
X=\cup_n M_n\cup M_0
\end{equation*}
where $\mu(M_0)=0$, the sets $M_n$ are compact with
$\mu(M_n)<\infty$, and $f\trestriction_{M_n}$ is continuous. If
$\mathcal{H}^h(f(M_0))>0$, then we are done. Otherwise, by the
fact that $\mathcal{H}^h$ is $G_\delta$\nobreakdash-regular as noted in
\tref{RogersRegular}, there is a Borel set $B\supset f(M_0)$ with
$\mathcal{H}^h(B)=0$. Note that
\begin{equation*}
Y=\cup_n f(M_n)\cup B
\end{equation*}
is, as the sets $f(M_n)$ are compact, a countable union of Borel
sets. If each set $f(M_n)$ would have a representation as
countable union of measurable sets with finite measure, then this would
also be true for $Y$. By our assumption on $Y$, this is not possible.
Hence, there is a set $M_n$ such that $f\trestriction_{M_n}\colon
M_n\surj f(M_n)$ is such that we can apply \tref{GeneralResult}. The statement follows.
\end{proof}

\begin{question} Can the set $N$ in \tref{MainInfinite} be chosen to be perfect? The problem lies in the case where $\mathcal{H}^h(f(M_n))>0$ implies that $n=0$.
\end{question}

\begin{proof}[Proof of \tref{Dimension}]
The proof is basically as the one of \tref{DimensionBusiness}. We just apply \tref{MainInfinite} instead of \tref{GeneralResult}.
\end{proof}

\section{Space-fillings}\label{Section space fillings}
We have been talking about surjections from one space onto another one. Here, we give a result concerning their existence. As blueprints for their constructions, the proofs of Theorem~1.3 in \cite{HajlaszTyson} and of Theorem~5.1 in \cite{Paper2} were used. See also Section~2 in \cite{AzzamSchul}.
\label{ChapSF}\begin{theorem}\tlabel{SpaceFilling}
Let $(X,d)$ be a locally compact metric space. Suppose $h$ is a Hausdorff function with $\lim_{t\to 0} h(t)=0$. Let $Y$ be any non-empty length-compact\footnote{This means that $Y$ is compact with respect to the path metric.} metric space. Suppose that $P\subset X$ is a non-empty perfect set. Then there exists  a compact, perfect set $P'\subset P$ and a continuous surjection $f\colon X\surj Y$ such that $\mathcal{H}^h(P')=0$ and $f(P')=Y$.
\end{theorem}
\toggletrue{SF}

\begin{proof}
We suppose that $Y$ is equipped with the path metric. Without loss of generality, we can assume that $\diam Y\leq 1$.

As building blocks of the continuous surjection, we will use mappings from annuli\footnote{Maybe, it would be more accurate to speak about balls instead of annuli. In each step, we modify the mapping in balls. However, if we look what additional part stays fixed when we go from one step to the next, then we obtain a system of annuli.} to paths. Thus, the ingredients for the construction of the continuous surjection are as follows: a system of annuli/balls, a system of paths, and bump functions. As the space $X$ is locally compact, it suffices to construct the continuous surjection in a compact ball such that the surjection is constant in a neighborhood of the boundary. Consequently, in what follows, we assume that $X$ is a compact ball and that $P$ has positive distance to the boundary. We will also tacitly assume that every chosen ball in $X$ is contained in this compact ball.\footnote{Fixing a point $x\in P$ and a closed ball $\overline{B}(x,r)$ that is compact, we note that $B(x,r/2)\cap P$ is uncountable. By the Cantor-Bendixson \tref{CantorBendixson}, we can replace the original perfect set by one in $\overline{B}(x,3r/4)\cap P$.}
We start by looking at the bump functions. Given $x_0\in X$ and $0<\delta<\varepsilon$, we can construct continuous functions $\eta\colon X\to [0,1]$ satisfying

\begin{itemize}
\item[(i)] $\supp \eta$ is a compact subset of $B(x_0,\varepsilon)$,
\item[(ii)] $\eta(x)=1$ for all $x\in B(x_0,\delta)$,
\item[(iii)] $\eta$ is Lipschitz.
\end{itemize}

Let us turn our focus to the system of paths in $Y$.
Since $Y$ is length-compact, we may find for each non-negative integer $n$ a finite set $Y_n=\{y_i^n\}_{i=1}^{k_n}$ with the property that each $y\in Y$ can be connected to a point in $Y_n$ by a path of length no greater than $2^{-n}$. Then $\cup_n Y_n$ is dense in $Y$.

For each integer $n\geq 1$, we may partition $Y_n$ into sets $\mathcal{C}(y_i^{n-1})$ so that if $y_j^n\in \mathcal{C}(y_i^{n-1})$, then there is a $1$\nobreakdash-Lipschitz path $\gamma^n_j\colon [0,2^{-(n-1)}]\to Y$ satisfying $\gamma^n_j(0)=y_i^{n-1}$ and $\gamma^n_j(2^{-(n-1)})=y^n_j$.

To each path, we need to assign a corresponding annulus. Let us mix the construction of the annuli and the mapping $f$. Actually, we will work with globally defined mappings and modify them in balls.

Let $f_0\colon X\to Y$ be the constant mapping $f_0(x)=y_1^0$ for all $x\in X$.

Fix some $x_0\in P$. As $P$ is perfect and non-empty, it is infinite, and so we may find a collection $\mathcal{C}(x_1^0)$ of $k_1$ distinct points $\{x_i^1\}_{i=1}^{k_1}\subset P$. Choose $\varepsilon_1>0$ so small and balls $\{B(x_i^1,\varepsilon_1)\}_{i=1}^{k_1}$ such that their centers have distance at least $3\varepsilon_1$ from each other (hence the balls are pairwise disjoint), and $k_1 h(2\varepsilon_1)<1$.
We can fix a number $0<\delta_1<\varepsilon_1$ and for each of the points $x_i^1$ a corresponding bump function $\eta_i^1\colon X\to [0,1]$ as\footnote{In the construction of the bump function, we let $\delta=\delta_1$, $\varepsilon=\varepsilon_1$, and $x_0=x_i^1$.} above.

As the collection $\{B(x_i^1,\varepsilon_1)\}_{i=1}^{k_1}$ consists of pairwise disjoint balls, we may define the mapping $f_1\colon X\to Y$ by
\begin{equation*}
f_1(x)=
\begin{cases}
\gamma_i^1\circ \eta_i^1(x) & x\in B(x_i^1,\varepsilon_1),
\\
f_0(x) & x\not\in \cup_{i=1}^{k_1} B(x_i^1,\varepsilon_1).
\end{cases}
\end{equation*}
Note that $f_1(x_i^1)=\gamma_i^1(1)=y_i^1$. It is easily checked that $f_1$ is Lipschitz continuous.

Now, we continue inductively. For each $n\in \field{N}$, we may find a set of points \mbox{$\{x_i^{n+1}\}_{i=1}^{k_{n+1}}\subset P$},
pairwise disjoint balls $B(x_i^{n+1},\varepsilon_{n+1})$, and a continuous mapping\footnote{For $\eta_i^{n+1}$, we choose $\delta=\delta_{n+1}$ and $\varepsilon=\varepsilon_{n+1}$ for the radii, and we further set $x_0=x_i^{n+1}$ in the construction of the bump function detailed before.} $f_{n+1}\colon X\to Y$ defined by
\begin{equation*}
f_{n+1}(x)=
\begin{cases}
\gamma_i^{n+1}(\frac{1}{2^n}\eta_i^{n+1}(x)) & x \in B(x_i^{n+1},\varepsilon_{n+1}),\\
f_n(x) & x\not\in \cup_{i}^{k_{n+1}} B(x_i^{n+1},\varepsilon_{n+1}),
\end{cases}
\end{equation*}
such that
\begingroup
\renewcommand{\theenumi}{\alph{enumi}}
\def\p@enumii{\theenumi}
\def\labelenumi{(\theenumi)}
\begin{enumerate}
\item $k_{n+1} h(2\varepsilon_{n+1})<\frac{1}{n+1}$,
\item $x_j^n\not\in B(x_i^{n+1}, \varepsilon_{n+1})\subset B(x_j^{n},\delta_n)$ if  $x_i^{n+1}\in \mathcal{C}(x_j^{n})$. Further, each $B(x_j^{n},\delta_n)$ contains at least two balls\footnote{The requirement that we need at least two balls is a small nuisance. We will need it in order to find a perfect set $P'$ that is blown up. However, it is no problem as we can always add additional paths.} $B(x_i^{n+1},\varepsilon_{n+1})$, the centers of the balls have distance bounded from below by $3 \varepsilon_{n+1}$, and each ball contains infinitely many points of $P$,
\item
for each $y_{n+1}\in \mathcal{C}(y_n^i)$, there is a $x^i_{n+1}\in \mathcal{C}(x_n^i)$,
\item $f_{n+1}(x)=f_n(x)$ for all $x\not\in \cup_{i=1}^{k_{n+1}}B(x_i^{n+1},\varepsilon_{n+1})$,
\item $f_{n+1}(x)=y_1^0$ if $x\not\in \cup_{i=1}^{k_1}B(x_i^1,\varepsilon_1)$,
\item $f_m(x_i^{n+1})=y_i^{n+1}$ for all integers $m\geq n+1$ and $i=1,\ldots, k_n$,
\item \label{CSFunction}$d_Y(f_{n+1}(x),f_n(x))\leq 2^{-n}$ for all $x\in X$,
\item $f_{n+1}$ is continuous.
\end{enumerate}
\endgroup
Point \eqref{CSFunction} above shows that $(f_n)_n$ is a Cauchy sequence of mappings in the supremum norm.
The sequence $(f_n)_n$ converges uniformly to a continuous function $f\colon X\to Y$.

We consider
\begin{equation*}
P':=\bigcap_{n\in \field{N}} \bigcup_{i=1}^{k_n} \overline{B}(x_i^n,\varepsilon_n)
\end{equation*}
and argue in the following that it is a as required in the \nref{SpaceFilling}.

Let us show the compactness of $P'$. As intersection of closed sets it is closed as well and as closed subset of a compact set, it is itself compact.

To conclude that $P'$ is not empty, we consider a sequence $(x_n)_n$ of centers of balls such that $x_{n+1}$ is in $B(x_n,\varepsilon_n)$. It follows that $x_{n+m}\in \overline{B}(x_n,\varepsilon_n)$ for all $m\in \field{N}$. By the compactness of $\overline{B}(x_n,\varepsilon_n)$, a subsequence of $(x_m)_m$ converges in $\overline{B}(x_n,\varepsilon_n)$. However, since the original sequence is a Cauchy sequence, the original sequence converges. Since $n$ was arbitrary, the existence of the limit in the intersection follows.

We continue by verifying that $P'$ is perfect. We have already verified that it is closed. For showing the last required property in the definition of perfect, we assume by contradiction that there is an open set $O$ that hits $P'$ but only in finitely many points. Let $x\in O\cap P'$. Then there is some $r>0$ such that $B(x,r)\subset O$ and $B(x,r)$ hits $P'$ only in finitely many, say $N$, points. There is a ball $\overline{B}(x_n,\varepsilon_n)$ that contains $x$ and lies in $B(x,r)$. We find a generation $n+m$ that has more than $N+1$ disjoint balls in $B(x_n,\varepsilon_n)$. Continuing similarly as when we showed that $P'$ is not empty, we obtain the existence of $N+1$ distinct points in $P'\cap B(x,r)$---a contradiction.

Now, we verify that $\mathcal{H}^h(P')=0$. Let $\delta>0$. We can choose $n$ so large that $2\varepsilon_n<\delta$. Thus
\begin{equation*}
\mathcal{H}_\delta^h(P')\leq k_n h(2\varepsilon_n)<\frac{1}{n}
\end{equation*}
for all large $n$. We see that $\mathcal{H}^h_\delta(P')=0$ and since $\delta>0$ was arbitrary, the claim follows.

Let us conclude the proof by showing that $f(P')=Y$.
Let $y\in Y$ and $(y_i^n)_n$ converging to $y$ such that $y_i^{n+1}\in \mathcal{C}(y_i^n)$. We pick for each $y_i^n$ a corresponding $x_n^i\in P'$ with $f(x_n^i)=y_n^i$ and $d(x_{n+1}^j,x_n^i)<\delta_n$. Note that the sequence $(x_n^i)_n$ converges to some $x\in P'$ with $f(x)=y$.
\end{proof}

\begin{remark} In the above \nref{SpaceFilling}, we cannot require $\mathcal{H}^h(P')=0$ simultaneously for all Hausdorff functions with $h(0)=0$, since in this case $P'$ is countable, see Corollary~3 on p.~67 in \cite{Rogers}.
\end{remark}

\begin{remark}[Assumes (\CH)]
According to Theorem~7.5 in \cite{Keesling}, assuming the Continuum Hypothesis and that $n$ is a natural number, there exists a separable metric space $X_n$ such that $0<\mathcal{H}^n(X_n)<\infty$ but with the property that there is no continuous map $f\colon X_n\to [0,1]$ that is onto.
\end{remark}

\section{Some space-fillings satisfying Luzin's condition~(N)}\label{Space fillings and condition N}
Previously, we have assumed that $X$ can be written as union of a certain amount of sets of finite measure, whereas $Y$ cannot. Actually, we have also excluded the case where $Y$ has $\sigma$\nobreakdash-finite measure. We have done so with good reason:

\begin{example}\tlabel{ex tan}
Let us consider the differentiable bijection $\tan\colon
(-\frac{\pi}{2},\frac{\pi}{2})\surj \field{R}$. Let us assume that
$(-\frac{\pi}{2},\frac{\pi}{2})$ and $\field{R}$ are equipped with
the Euclidean distance and Lebesgue measure. Hence
$(-\frac{\pi}{2},\frac{\pi}{2})$ has finite measure and
$\field{R}$ is $\sigma$\nobreakdash-finite. Note that as locally
Lipschitz continuous function, the tangent $\tan$ satisfies Luzin's
condition~(N).
\end{example}

The following \nref{NullDim} enlightens that different dimensions do not always force sets to been blown up. Note that there are compact $0$\nobreakdash-dimensional spaces that are not countable.

\begin{example}\tlabel{NullDim} Assume that $(X,d_X, \mathcal{H}^0)$ is
$0$\nobreakdash-dimensional, and $(Y,d_Y,\nu)$ is a metric measure space. If
$f\colon X\surj Y$ is a surjection, then it satisfies Luzin's
condition~(N).
\end{example}

\begin{proof}
As $\mathcal{H}^0$ is the counting measure, the only subset of $X$
with measure zero is the empty set, which is mapped onto the empty
set as well.
\end{proof}

\begin{remark}
Why does above \nref{NullDim} not violate our results? Let $\nu=\mathcal{H}^h$ as specified in our results. As $f$ is a surjection, we have $\card{Y}\leq \card{X}$ violating the conditions of \tref{GeneralResult} and \tref{GeneralResultNonSeparable}, and if $X$ is countable also of \tref{Compact target}. In the remaining case, $Y$ is assumed to be compact and thus separable. By Section~2 in \cite{Freiwald}, $\card(Y)\in \{\aleph_0,2^{\aleph_0}\}$ --- both alternatives are not in the scope of \tref{Compact target}.
\end{remark}
\ifthenelse{\equal{\nref{ex tan}}{\nref{NullDim}}}{The two preceding
\nuref{ex tan}s~\pref{ex tan} and \pref{NullDim}}%
{\tref{ex tan} and \tref{NullDim}} are quite special. In \ifthenelse{\equal{\nref{ex tan}}{\nref{NullDim}}}{the first \nref{ex tan}}{\tref{ex tan}}, the domain and the target have the same dimension, while in \ifthenelse{\equal{\nref{ex tan}}{\nref{NullDim}}}{the \specialref{NullDim}, }{\tref{NullDim}, }%
the domain is zero-dimensional and has therefore no non-trivial sets of measure zero. Our goal is to construct space-fillings whose domain has dimension larger than zero, and whose dimension of the target does not necessarily agree with the one of the domain. In most constructions, we assume the Continuum Hypothesis for the existence of certain sets in the domain.

Before we can start with constructing such sets, we need to deal with cardinalities of families of certain subsets in metric spaces.

\begin{lemma}\tlabel{cardinalities} Let $(X,d)$ be a separable metric space with at least\footnote{The conclusion shows that in this case, we actually have exactly $2^{\aleph_0}$ elements.} $2^{\aleph_0}$ elements. Assume that $h$ is a Hausdorff function with $h(0)=0$. Let
\begin{align*}
\mathcal{O}&:=\{O\subset X:\; \text{$O$ open}\},\\
\mathcal{C}&:=\{C\subset X:\; \text{$C$ closed}\},\\
\mathcal{E}&:=\{E\subset X:\; \text{$E$ is a $G_\delta$\nobreakdash-set that is $\sigma$\nobreakdash-finite with respect to $\mathcal{H}^h$}\}.
\end{align*}
Then $\card \mathcal{O}=\card \mathcal{C} =\card \mathcal{E}=2^{\aleph_0}$.
\end{lemma}

\begin{proof}
By Theorem~\ignoreSpellCheck{XIV}.3.1 in \cite{Kuratowski}, the cardinality of the family of all open sets is bounded from above by $2^{\aleph_0}$. The same upper bound applies to the closed sets. Since points are closed, both families $\mathcal{O}$ and $\mathcal{C}$ have cardinality $2^{\aleph_0}$. As points are $G_\delta$\nobreakdash-sets with zero measure, it suffices to prove that $\card \mathcal{E}\leq 2^{\aleph_0}$ to obtain that $\card \mathcal{E}=2^{\aleph_0}$. For each set $E\in \mathcal{E}$ we find a sequence $O_i$ of open sets whose intersection is $E$. Thus the cardinality of $\mathcal{E}$ is bounded from above by the cardinality of all sequences of real numbers. Cardinal arithmetics, see for example Sections~3 and 4 in \cite{Jech}, gives the wished upper bound.
\end{proof}

\begin{lemma}\tlabel{exactNumberOfPerfectSet}
Assume $(X,d,\mathcal{H}^h)$ is compact and satisfies \Hinf.
Then $X$ contains exactly $2^{\aleph_0}$ perfect sets of non-$\sigma$\nobreakdash-finite measure. Moreover, we can find $2^{\aleph_0}$ perfect sets of non\nobreakdash-$\sigma$\nobreakdash-finite measure that are pairwise disjoint.
\end{lemma}

\begin{proof}
Let us first determine the cardinality of the family of the closed sets that do not have $\sigma$\nobreakdash-finite $\mathcal{H}^h$\nobreakdash-measure. Note that $X$ is separable, thus the desired upper bound for the cardinality of the collection of the closed sets of \tref{cardinalities} applies.
From \tref{perfect sets in the compact case}, we know of the existence of the required amount of pairwise disjoint compact sets with the right size.

Let us now tackle the cardinality of the perfect sets. Since perfect sets are closed, the same upper bound applies. Separable metric spaces have a countable basis, and thus Cantor-Bendixson's \tref{CantorBendixson} gives the desired amount of perfect sets.
\end{proof}

The following construction is inspired by an example by A.~S.~Besicovitch, see Chapter~II in \cite{BesicovitchRarified}; parts of Besicovitch's example are also used in \tref{BesicovitchsSet}:
\begin{proposition}[Assumes (\CH)] \tlabel{BesicovitchMartin} Assume that $(X,d,\mathcal{H}^h)$ is compact, satisfies \Hinf, and that (\CH) holds. Then there exists a set $G\subset X$ such that
\begingroup
\renewcommand{\theenumi}{\alph{enumi}}
\def\p@enumii{\theenumi}
\def\labelenumi{(\theenumi)}
\begin{enumerate}
\item \label{Beicovitch Cardinality} the cardinality of $G$ is $2^{\aleph_0}$,
\item \label{Besicovitch Intersection} if $E\subset G$ is $\sigma$\nobreakdash-finite with respect to $\mathcal{H}^h$, then $G\cap E$ is countable,
\item $G$ is non-$\sigma$\nobreakdash-finite with respect to $\mathcal{H}^h$,
\item each subset $A\subset G$ is $\mathcal{H}^h$\nobreakdash-measurable.
\end{enumerate}
\endgroup
\end{proposition}

\begin{proof}
We set
\begin{align*}
\mathcal{P}&:=\{P\subset X:\; \text{$P$ is perfect but not $\sigma$\nobreakdash-finite with respect to $\mathcal{H}^h$}\},\\
\mathcal{E}&:=\{E\subset X:\; \text{$E=\cap_{n\in \field{N}}O_n$, $O_n$ open, and $E$ is $\sigma$\nobreakdash-finite with respect to $\mathcal{H}^h$ }\}.
\end{align*}
Using \ifthenelse{\equal{\nref{cardinalities}}{\nref{exactNumberOfPerfectSet}}}{%
\nuref{cardinalities}s~\pref{cardinalities} and \pref{exactNumberOfPerfectSet}}%
{\tref{cardinalities} and \tref{exactNumberOfPerfectSet}}, we obtain that
\begin{equation*}
\card \mathcal{P}=\card \mathcal{E}=2^{\aleph_0}.
\end{equation*}
Let us denote the ordinal corresponding to the real numbers by $\ordReals$. There exist bijections between $\ordReals$ and $\mathcal{P}$, and between $\ordReals$ and $\mathcal{E}$, respectively. Given $i\in \ordReals$, we denote by $P_i$ and $E_i$ the corresponding element in $\mathcal{P}$ and $\mathcal{E}$, respectively. Let us choose an arbitrary point $x_1$ out of the set $P_1\setminus E_1$. Since $P_1$ is not $\sigma$\nobreakdash-finite, but $E_1$ is, such a point certainly exists.

Take $i_0\in \ordReals$ and assume that for each $i<i_0$, we have already chosen a point $x_i\in P_i\setminus \cup_{k=1}^i E_k$, different from all the previously chosen points $x_k$ for $k<i$. We choose now a point $x_{i_0}\in P_{i_0}\setminus \cup_{k=1}^{i_0} E_k$ different from all the already chosen points $x_k$. Arguing similarly as in the choice of $x_1$, such a point $x_{i_0}$ certainly exists. We collect the points in the following set
\begin{equation*}
G:=\{x\in X:\, \text{there exists $i\in \ordReals$ such that $x=x_i$}\}.
\end{equation*}
Before proving that $G$ has the required properties, let us note that if $P\in \mathcal{P}$, then there is an index $i\in \ordReals$ with $P=P_i$. Thus there exists a point $x_i\in G\cap P$ implying that $G\cap P$ is not empty.
\begin{itemize}
\item[(a)] By construction.
\item[(b)] Assume that $E$ is $\sigma$\nobreakdash-finite with respect to $\mathcal{H}^h$, and we suppose first that it is a countable intersection of open sets. In this case, there is an index $i_0\in \ordReals$ such that $E=E_{i_0}$. For all $i>i_0$, we have
    \begin{equation*}
    x_i\in P_i\setminus \cup_{k=1}^i E_k\subset P_i\setminus E_{i_0}.
    \end{equation*}
    Hence at most the points $x_1,x_2,\ldots,x_{i_0}$ lie in $E_{i_0}$ and $G$. It follows that $G\cap E=G\cap E_{i_0}$ is countable.

    If the $\sigma$\nobreakdash-finite set $E$ is not a countable intersection of open sets, then we write $E=\cup_{n\in \field{N}}E_n$, where the sets $E_n$ are measurable sets with finite measure. By \tref{RogersRegular}, the Hausdorff measure is $\mathcal{G}_\delta$\nobreakdash-regular. Thus the sets $E_n$ are contained in sets $F_n$ that can be written as countable intersection of open sets. Hence, by our considerations above, each $F_n\cap G$ is countable,
    and consequently this is true for each $E_n\cap G$ and finally for $E\cap G$.
\item[(c)] If $G$ would be $\sigma$\nobreakdash-finite, then by \eqref{Besicovitch Intersection}, it would
    be countable.
    This would contradict \eqref{Beicovitch Cardinality}.
\item[(d)] Suppose $A\subset G$. Let $F\subset X$ be an arbitrary subset. We have to show that
    \begin{equation*}
    \mathcal{H}^h(F)\geq \mathcal{H}^h(F\cap A)+\mathcal{H}^h(F\setminus A)
    \end{equation*}
    holds. If $\mathcal{H}^h(F)=\infty$ or $\mathcal{H}^h(F)=0$, the inequality follows. Thus, we may assume that $0<\mathcal{H}^h(F)<\infty$. It follows that $\mathcal{H}^h(F\cap A)<\infty$. But this implies by \eqref{Besicovitch Intersection} that $F\cap A\subset G$
    is countable,
    and our assumptions on $h$ imply that $\mathcal{H}^h(F\cap A)=0$. But then the inequality holds as well.
\end{itemize}
\end{proof}

\begin{theorem}[Assumes (\CH)]\tlabel{BesiSF} Suppose $(X,d,\mathcal{H}^h)$ is compact and satisfies \Hinf. We further stipulate that (\CH) holds.

Assume that $(Y,\nu)$ is a measure space such that $\card Y=2^{\aleph_0}$, and $\nu$ is such that points have measure zero.

Then there exists a  measurable surjection $f\colon X\surj Y$ such that whenever $N\subset X$ is $\sigma$\nobreakdash-finite with respect to $\mathcal{H}^h$, then $\nu(f(N))=0$.
\end{theorem}

\begin{proof}
Choose $G\subset X$ as in \tref{BesicovitchMartin}. According to \tref{uncountablePartition}, we can write $G$ as union of $2^{\aleph_0}$ many, uncountable sets $A_\lambda$, $\lambda\in \field{R}$, that are pairwise disjoint. We denote by $F\colon \field{R}\surj Y$ a bijection and fix a point $y_0\in Y$. We define $f\colon X\to Y$ by $f(x)=F(\lambda)$ if $x\in A_\lambda$, and we set\footnote{By Howroyd's \tref{ExistenceOfSubsets}, we see that $G\not=X$.} $f(x)=y_0$ if $x\in X\setminus G$. By construction, $f$ is a surjection.

Remember that each subset of $G$ is measurable. The preimage of any subset of $Y$ is either a subset of $G$ or the union of $X\setminus G$ and a subset of $G$. In both cases, the preimage is measurable. This shows that $f$ is measurable.

A set $N$ of measure zero splits in a part lying in $G$ and one lying in the complement of $G$. The part in the complement is mapped to the point $y_0$; its measure is zero. The part inside $G$ has countably many points and is thus mapped to a set with the same cardinality upper bound. This gives the claim.
\end{proof}

We may ask if the dichotomy between Luzin's condition~(N) and space-fillings is still valid if we
do not require the mapping to be measurable. To answer this question in \tref{filling and N}, we first introduce a peculiar set constructed by Besicovitch, which has been the main source of inspiration for the construction of the set in \tref{BesicovitchMartin}.  Here again, we assume that the Continuum Hypothesis (\CH) holds.

\begin{example}[Assumes (\CH)]\tlabel{BesicovitchsSet}
Besicovitch
constructs in \cite[Chapter~II]{BesicovitchRarified} under the continuum
hypothesis a set $H\subset [0,1]^2$, which has, amongst others,
the following two properties
\begin{enumerate}
\item[(a)] the exterior plane measure of $H$ is $1$,
\item[(b)] any subset of $H$ of plane measure zero is a countable set.
\end{enumerate}
For example from Theorem~264I in \cite{Fremlin}, we know that the
Hausdorff measure $\mathcal{H}^2$ and the Lebesgue measure
$\mathcal{L}^2$ have the same measurable sets, the same null
sets, and agree up to a factor on the $\sigma$\nobreakdash-algebra of measurable sets. We claim that $\mathcal{H}^2(H)>0$. Since $\mathcal{H}^2$ is
a $\mathcal{G}_\delta$\nobreakdash-regular Borel measure by \tref{RogersRegular}, there exists a positive constant $C$ and a Borel set $B$ with $H\subset B$
and
\begin{equation*}
\mathcal{H}^2(H)=\mathcal{H}^2(B)=C\leb^2(B)\geq C\leb^2(H)=C>0.
\end{equation*}
\end{example}

In the following
\nref{filling and N}, we construct (under the Continuum Hypothesis) a non-measurable space-filling that satisfies Luzin's condition~(N). The main difference to the mapping constructed in \tref{BesiSF} is that here, the domain has finite measure.
\begin{example}[Assumes (\CH)\tlabel{filling and N}]
If we have a surjection $f\colon H\surj [0,1]^3$ (the existence follows since
by the continuum hypothesis, $H$ and $[0,1]^3$ have the same
cardinality), then it satisfies Luzin's condition~(N) with respect to $\mathcal{H}^2$ on domain and target, since the only
sets of measure zero are countable and are thus mapped on
countable sets in $[0,1]^3$ as well.

We can extend the mapping $f$ to a surjection $F\colon [0,1]^2\to
[0,1]^3$ by letting $F(x)=0$ whenever $x\not\in H$ and $F(x)=f(x)$
otherwise. \tref{MainInfinite} tells us
that $F$ cannot be measurable.
\end{example}

Measurability of a mapping depends to a certain degree on the topology of the target. The following \nref{trivialTopology} contains a version of \tref{filling and N} where the topology of the target is trivial. It also shows that in \tref{StandaAbstract}, we cannot get rid of the assumption of the existence of the desired disjoint Borel sets.
\begin{example}[Assumes (\CH)]\tlabel{trivialTopology}
Let $(Y,\mathfrak{T})$ be a topological space, where the topology $\mathfrak{T}$ is
given by $\mathfrak{T}=\{Y,\emptyset\}$. Suppose $\nu$ is a non-$\sigma$\nobreakdash-finite
Borel measure without atoms on $Y$ and $\card(Y)=\card(\field{R})$.

Let $H$ be the set defined in \tref{BesicovitchsSet}. It is uncountable
and by the Continuum Hypothesis, its cardinality is the same as that
of $[0,1]^2$. We find a bijection $\tilde{f}\colon H\surj Y$. Fix a point $y_0\in Y$
and let $f\colon [0,1]^2\surj Y$ be defined by $f(x)=\tilde{f}(x)$ if $x\in H$ and
$f(x)=y_0$ otherwise.

We assume that $[0,1]^2$ is equipped with the standard metric and measure. Let
$N\subset [0,1]^2$ be a set of measure zero. Then
\begin{equation*}
N_H:=N\cap H
\end{equation*}
is of measure zero as well, and hence it is countable. Thus $f(N_H)$ is countable. The
set $N_{[0,1]^2\setminus H}:=N\cap ([0,1]^2\setminus H)$ is mapped to $y_0$. So $f(N_{[0,1]^2\setminus H})$ and finally $f(N)$ are countable. Since $\nu$ has no atoms, $f(N)$ has measure zero. Thus
$f$ is a measurable space-filling satisfying Luzin's condition~(N).
\end{example}

\section{Applications}\label{Applications}
For example by Proposition~423B in \cite{Fremlin}, Polish spaces are analytic. Choosing for $f$ the identity in \tref{GeneralResult} and by way of contradiction, we obtain the following result:
\begin{corollary}[Corollary of \tref{GeneralResult}]
Let $(X,d,\mu)$ be a complete, separable metric measure space, where $\mu$ is Borel. Suppose $X$ is $\sigma$\nobreakdash-finite with respect to $\mu$. We assume further that $h$ is a continuous Hausdorff function of finite order with $h(0)=0$, and that $\mathcal{H}^h$ is absolutely continuous with respect to $\mu$. Then $X$ is $\sigma$\nobreakdash-finite with respect to $\mathcal{H}^h$.
\end{corollary}

From our results, we can obtain rigidity results:
\begin{theorem}\tlabel{Rigidity}
Let $(X,\mu)$ be an analytic Hausdorff space such that $\mu$ is Borel with \mbox{$\mu(X)<\infty$}. Suppose $Y$ is a complete and separable metric space and $h$ a continuous Hausdorff function with $h(0)=0$ and one of the following is satisfied:
\begin{itemize}
\item[(a)] $h$ is of finite order
\item[(b)] $Y$ has finite structural dimension
\item[(c)] $Y$ is ultrametric.
\end{itemize}
Suppose that $f\colon X\to Y$ is Borel measurable and satisfies Luzin's condition~(N) in the sense that $\mu(N)=0$ implies that $\mathcal{H}^h(f(N))=0$. Then $f(X)$ is $\sigma$\nobreakdash-finite with respect to $h$.
\end{theorem}

\begin{proof}
Under the given conditions, $Y$ is analytic, see for example Proposition~423B in \cite{Fremlin} and by Lemma~423G in the same source, we know that $f(X)$ is analytic. Now, $f$ is surjective onto its image. If it would not be $\sigma$\nobreakdash-finite with respect to $\mathcal{H}^h$, then
\tref{GeneralResult} would imply that $f$ violates Luzin's condition~(N). Thus the claim follows.
\end{proof}

We cite Theorem~1.3 in \cite{CHM}:
\begin{theorem}\tlabel{CHMHomeo}
Let $f\in W^{1,n-1}_\textnormal{loc}((-1,1)^n,\field{R}^n)$ be a homeomorphism\footnote{onto $f((-1,1)^n)$}.
Then for almost every $y\in (-1,1)$ the mapping $f\trestriction_{(-1,1)^{n-1}\times \{y\}}$ satisfies
the (\mbox{$(n-1)$}\nobreakdash-di\-men\-sion\-al) Luzin condition~(N), i.e., for every $A\subset (-1,1)^{n-1}\times
\{y\}$, \mbox{$\mathcal{H}^{n-1}(A)=0$} implies $\mathcal{H}^{n-1}(f(A))=0$.
\end{theorem}

Combining above \nref{CHMHomeo} with \tref{Rigidity} and further exhausting the sets $(-1,1)^{n-1}\times\{y\}$ with compact sets, we obtain the following result:
\begin{corollary}
Let $f\in W^{1,n-1}_\textnormal{loc}((-1,1)^n,\field{R}^n)$ be a homeomorphism. Then for almost every $y\in (-1,1)$, we obtain that $f((-1,1)^{n-1}\times \{y\})$ has dimension bounded from above by $n-1$.
\end{corollary}
\iftoggle{SF}{}{\todonote{I promised to talk about space-fillings, but did not}}

\bibliographystyle{alpha}
\bibliography{Dichotomy}
\end{document}